\pdfoutput=1
\RequirePackage{ifpdf}
\ifpdf 
\documentclass[pdftex]{sigma}
\else
\documentclass{sigma}
\fi

\numberwithin{equation}{section}

\newtheorem{thm}{Theorem}[section]
\newtheorem{cor}[thm]{Corollary}
\newtheorem{lem}[thm]{Lemma}
\newtheorem{prop}[thm]{Proposition}
\newtheorem{prop-defn}[thm]{Proposition-Definition}

\theoremstyle{definition}
\newtheorem{rk}[thm]{Remark}
\newtheorem{ex}[thm]{Example}
\newtheorem{defn}[thm]{Definition}
\newcommand{\C}{\mathbb C}

\newcommand{\R}{\mathbb R}
\newcommand{\B}{\mathbb B}
\newcommand{\Pp}{\mathbb P}

\newcommand{\Z}{\mathbb Z}
\newcommand{\M}{\mathbb M}
\newcommand{\N}{\mathbb N}

\newcommand{\E}{\mathbb E}

\newcommand{\Cl}{{\rm C}\ell}

\usepackage{slashed}

\newcommand{\Exp}{{\mathbf E}{\rm xp}}
\newcommand\ind{{\rm 1\kern-.30em I}}
\def \bui#1#2{\mathrel{\mathop{\kern 0pt#1}\limits^{#2}}}

\newcommand{\biindice}[3]%
{%

\begin{array}[t]{c}
{\displaystyle #1}\\
{\scriptstyle #2}\\
{\scriptstyle #3}
\end{array}

}

\begin{document}

\allowdisplaybreaks

\renewcommand{\thefootnote}{}

\newcommand{\arXivNumber}{2303.04013}

\renewcommand{\PaperNumber}{010}

\FirstPageHeading

\ShortArticleName{A Pseudodifferential Analytic Perspective on Getzler's Rescaling}

\ArticleName{A Pseudodifferential Analytic Perspective\\ on Getzler's Rescaling\footnote{This paper is a~contribution to the Special Issue on Differential Geometry Inspired by Mathematical Physics in honor of Jean-Pierre Bourguignon for his 75th birthday. The~full collection is available at \href{https://www.emis.de/journals/SIGMA/Bourguignon.html}{https://www.emis.de/journals/SIGMA/Bourguignon.html}}}

\Author{Georges HABIB~$^{\rm ab}$ and Sylvie PAYCHA~$^{\rm c}$}

\AuthorNameForHeading{G.~Habib and S.~Paycha}

\Address{$^{\rm a)}$~Department of Mathematics, Faculty of Sciences II, Lebanese University,\\
\hphantom{$^{\rm a)}$}~P.O.~Box, 90656 Fanar-Matn, Lebanon}
\EmailD{\href{ghabib@ul.edu.lb}{ghabib@ul.edu.lb}}
\Address{$^{\rm b)}$~Universit\'e de Lorraine, CNRS, IECL, France}
\URLaddressD{\url{https://iecl.univ-lorraine.fr/membre-iecl/habib-georges/}}

\Address{$^{\rm c)}$~Institut f\"ur Mathematik, Universit\"at Potsdam, Campus Golm,\\
\hphantom{$^{\rm c)}$}~Haus 9, Karl-Liebknecht-Str. 24-25, 14476 Potsdam, Germany}
\EmailD{\href{paycha@math.uni-potsdam.de}{paycha@math.uni-potsdam.de}}
\URLaddressD{\url{https://www.math.uni-potsdam.de/~paycha/paycha/Home.html}}

\ArticleDates{Received March 08, 2023, in final form January 11, 2024; Published online January 30, 2024}

\Abstract{Inspired by Gilkey's invariance theory, Getzler's rescaling method and Scott's approach to the index via Wodzicki residues, we give a localisation formula for the $\Z_2$-graded Wodzicki residue of the logarithm of a class of differential operators acting on sections of a spinor bundle over an even-dimensional manifold. This formula is expressed in terms of another local density built from the symbol of the logarithm of a limit of rescaled differential operators acting on differential forms. When applied to complex powers of the square of a~Dirac operator, it amounts to expressing the index of a Dirac operator in terms of a local density involving the logarithm of the Getzler rescaled limit of its square.}


\Classification{58J40; 47A53; 15A66}

\begin{flushright}
{\it Dedicated to J.P.~Bourguignon for his 75th birthday}
\end{flushright}

\renewcommand{\thefootnote}{\arabic{footnote}}
\setcounter{footnote}{0}

\section{Introduction}
On a closed Riemannian manifold $(M^n,g)$, the algebra $\Psi_{\rm cl}(M, E)$ of classical pseudodifferential operators acting on the smooth sections of a finite rank vector bundle $E$ over $M$, admits a unique (up to a multiplicative factor) trace, called the Wodzicki \cite{W} or the noncommutative residue, built from a residue density defined as follows. Given $Q$ in $\Psi_{\rm cl}(M,E)$, the residue of $Q$ is the integral over $M$ of the residue density $\omega_Q^ {\rm Res}(x):= {\rm res}(\sigma(Q)(x, \cdot)) {\rm d}x^1\wedge \dots\wedge {\rm d}x^n$ defined in (\ref{eq:Wodzres}) with
\[
{\rm res}(\sigma(Q)(x, \cdot)):=\frac{1}{(2\pi)^n} \int_{|\xi|=1}{\rm tr}^E(\sigma_{-n}(Q) (x, \xi)) {\rm d}_S\xi .
\]
Here, $n$ is the dimension of $M$, ${\rm tr}^E$ stands for the fibrewise trace on $\operatorname{End}(E)$, $(x,\xi)$ is an element in $T^*M$, and $\sigma_{-n}(Q)(x, \xi)$ is the $(-n)$-th homogeneous part of the symbol at $(x, \xi)$. The Wodzicki residue extends beyond classical pseudodifferential operators to the logarithm~$\log_\theta (Q)$ (see Appendix~\ref{appendix2} for the precise definition) of a pseudodifferential operator $Q$ with Agmon angle~$\theta$ (see \eqref{eq:reslog}), giving rise to the logarithmic residue $\omega_{\log_\theta (Q)}^ {\rm Res}(x)$. The Wodzicki residue is local in so far as it is expressed as the integral on $M$ of a volume form involving the $(-n)$-homogeneous component of the symbol. So it comes as no surprise that the index of the Dirac operator can be expressed in terms of the residue. For a $\Z_2$-graded vector bundle $E=E^+\oplus E^-$, the index of an elliptic odd operator $\slashed D^+\colon C^\infty(M, E^+)\to C^\infty(M, E^-)$ with formal adjoint $\slashed D^-=\bigl(\slashed D^+\bigr)^*$ can be written \cite{Sc}
\[
{\rm Index}\bigl(\slashed D^+\bigr)=-\frac{1}{2}{\rm sres}\bigl(\log _\theta\bigl(\slashed D^2\bigr)\bigr)=-\frac{1}{2} \int_M \omega_{\log_\theta (\slashed D^2)}^{{\rm sRes}}(x),
\]
where $\slashed D:= \left[\begin{smallmatrix}0&\slashed D^-\\
\slashed D^+& 0\end{smallmatrix}\right]$ acting on $E^+\oplus E^-$, so that $\slashed D^2=\slashed D^-\slashed D^+\oplus\slashed D^+\slashed D^-$ and $\theta=\pi$. The graded residue ``sres'' is defined in the same way as the residue with the fibrewise trace on $\operatorname{End}(E)$ replaced by the $\Z_2$-graded trace and $\omega_{\log _\theta(\slashed D^2)}^ {\rm sRes}(x):= {\rm sres}\bigl(\sigma\bigl(\log _\theta\bigl(\slashed D^2\bigr)\bigr)(x, \cdot)\bigr){\rm d}x^1\wedge\dots\wedge {\rm d}x^n$.

Inspired by the approach adopted in \cite{Sc}, we revisit Geztler's rescaling in the context of~index theory in the light of the logarithmic Wodzicki residue. For a class of differential operators acting on spinors which includes $\slashed D^2$, we express the logarithmic residue density evaluated at a~point~$p$~in~$M$ in terms of another local density
\smash{$\widetilde \omega_{\log_\theta (\widetilde\Pp^{\lim})}^{\rm sres}(x)$} (see formula (\ref{eq:tildeomega}))
involving a~limit~$\tilde\Pp^{\lim}$ as the parameter $\lambda$ goes to zero of a family of operators \smash{$\tilde \Pp_\lambda^{\rm Ge}$} built from the original~one by rescaling it at the point $p$ (see equation~(\ref{eq:PPlambdaGeLambda})).
In this sense, equation~(\ref{eq:locform}) at the limit as $\lambda $ tends to zero, can be viewed as a localisation formula of the logarithmic residue at point $p$. For this purpose, we first single out a class of differential operators acting on smooth sections of a vector bundle~$E$, which we call {\it geometric}
with respect to a metric $g$, whose coefficients written in some local trivialisation are geometric sections (Definition~\ref{defn:geometricdiffop}). We consider polynomials in the jets of the vielbeins for the metric $g$ (Definition~\ref{defn:geometricpolE}) and inspired by Gilkey \cite{G}, we define their {\it Gilkey order} (at a point $p$) (see equation~\eqref{eq:polynomialsvielbeins}) to be the order of those jets. We call a differential operator geometric if its coefficients are geometric polynomials whose Gilkey order obeys a~compatibility condition involving the order of the operator, see equation~(\ref{eq:geomopGi}). Geometric differential operators enjoy nice transformation properties under local contractions~$({\mathfrak f}_\lambda)_{\lambda\in [0,1]}$ along local geodesics defined by means of exponential geodesic normal coordinates (see equation~\eqref{eq:flambdaUplambda} for the definition). Indeed a geometric differential operator with respect to $g$ transforms to one with respect to $g_\lambda$ (Proposition~\ref{pro:geometricopes}), where {\color{black}$g_\lambda=\lambda^{-2}{\mathfrak f}_\lambda^*g$}. This transformed metric can be viewed as the pull-back metric under the canonical projection ${\hat\pi\colon \M \to M}$ of the deformed manifold $\M$ via a deformation to the normal cone to $p$, see equation~(\ref{eq:pistarg}).

We first consider the bundle $E=\Lambda T^*M$. From a differential operator $P$ in $\Psi_{\rm cl}(M, \Lambda T^*M)$ acting on differential forms, we define a family of operators $\widetilde \Pp_\lambda^{\rm Ge}:=\lambda^{{\rm ord}(P)} U_\lambda^\sharp\mathfrak{f}_\lambda^\sharp P$ (see \eqref{eq:PPlambdaGeLambda}) using notations borrowed from \cite{vEY}, which are obtained under the combined action of the {\color{black} contractions~${\mathfrak f}_\lambda$} mentioned previously and the so-called {\it Getzler map} $U_\lambda$ that acts on tensors, see Definition~\ref{getzlerrescaling}.
 We call a geometric differential operator $P $ {\it rescalable} if $\widetilde \Pp_\lambda^{\rm Ge}$ admits a limit $\widetilde \Pp^{\lim}$ when ${\lambda\to 0}$ (Definition~\ref{defn:limitdiffope}). In Proposition~\ref{limitdiffope}, we give a necessary and sufficient condition for the rescalability of a geometric differential operator in $\Psi_{\rm cl}(M, \Lambda T^*M)$ and show that the coefficients of the limit $\widetilde \Pp^{\lim}$ are polynomial expressions in the jets of the Riemannian curvature tensor.\looseness=-1

A first result is the localisation formula (\ref{eq:loctildeomegadiff}) for a differential operator $P$ acting on differential forms. It involves a local $n$-degree form $\widetilde \omega_Q^ {\rm Res}(x)$, inspired by Scott's proof of the index theorem \cite[Section~3.5.3]{Sc} and defined for operators $Q$ in $\Psi_{\rm cl} (M, \Lambda T^*M)$ as (see equation~\eqref{eq:omegatildeRes}):
 \begin{equation}\label{eq:tildeomega}
 \widetilde \omega^{\rm Res}_Q(x):= \frac{1}{(2\pi)^n} \int_{S_x^*U_p }\left[ \sigma_{-n} (Q) (x, \xi) {\bf 1}_x\right]_{[n]} {\rm d}_S\xi.
 \end{equation}
 Here $U_p$ is a local exponential neighborhood of a point $p$ in $M$, $S_x^*U_p$ is the unit sphere in the cotangent space $T_x^*U_p$ at point $x$ and the integrand is the degree $n$-part of the differential form $\sigma_{-n} (Q) (x, \xi) {\bf 1}_x$.
 When restricted to operators in the range of a Clifford map, the local form~$\widetilde \omega^{\rm Res}_Q$ is proportional to the Wodzicki residue density (Corollary~\ref{cor:relationomegatilde}) and therefore becomes a global form. It further extends to logarithmic pseudodifferential operators $Q=\log_\theta(P)$ for a differential operator $P$ with Agmon angle $\theta$. In Proposition~\ref{prop:pullbacklocalform}, we show that the local $n$-form $\widetilde \omega^{\rm Res}_{\log _\theta(P)}$ at the point $\mathfrak{f}_\lambda(x)$ is the local $n$-form associated to the pull-back operator \smash{$U_\lambda^\sharp\mathfrak{f}_\lambda^\sharp (\log _\theta(P))$} at the point $x$.
 If moreover $P$ is rescalable, taking the limit as $\lambda $ tends to zero yields the {\it localisation formula}
 \begin{equation}\label{eq:loctildeomegadiff}
 \widetilde	 \omega_{\log_\theta (P)}^{{ {\rm Res}}}(p)= \widetilde\omega_{\log_\theta (\widetilde \Pp^{\lim})}^{{ {\rm Res}}}(x), \qquad \forall x\in U_p.
 \end{equation}
This formula can be applied to the Hodge Laplacian which is a geometric and rescalable operator (see Example~\ref{ex:Hodgelaplacian}).

We then consider the case of a spinor module $E=\Sigma M$ when $M$ is a spin manifold of even dimension. To define rescalability of geometric differential operators in $\Psi_{\rm cl}(M, \Sigma M)$ we use the identification $\Cl(TM)\otimes \C\simeq \operatorname{End}(\Sigma M),$ and the {\it Clifford map} $ c^g \colon \Cl( T M) \longrightarrow \operatorname{End}(\Lambda T ^*M)$ (see \eqref{eq:cliffend}) which sends an element of the Clifford algebra $\Cl( T M)$ on the tangent bundle to an endomorphism of $\Lambda T^*M$. We call a geometric differential operator $P$ in $\Psi_{\rm cl}(M, \Sigma M)$ rescalable {if~$c^g(P)$ defined in formula (\ref{eq:cgp}) by applying $c^g $ to the coefficients of the differential operator~$P$}, is rescalable in $\Psi_{\rm cl}(M, \Lambda T^*M)$. We further give a necessary and sufficient condition for the rescalability of geometric differential operators in $\Psi_{\rm cl}(M, \Sigma M)$, see Proposition~\ref{prop:conditionrescalibilityClifford}.\footnote{ In this part of the work, we consider even-dimensional spin manifolds, however our study extends to manifolds with a spin$^c$ structure. Indeed the construction relies on the identification $\Cl(TM)\otimes \C\simeq \operatorname{End}(\Sigma M)$ which can be extended to manifolds with a spin$^c$ structure, in which case we have the identification $\Cl(TM)\otimes \C\simeq \operatorname{End}(E)$ \cite[Theorem~2.13]{BHMMM} with $E$ a vector bundle isomorphic to the spinor bundle of the spin$^c$ bundle. 
 For simplicity, we restrict ourselves to spin manifolds.}

It follows from Proposition~\ref{prop:tildeomegavsomega} that for a differential operator $P\in \Psi_{\rm cl}(M, \Sigma M)$ with Agmon angle $\theta$, the form $\widetilde \omega^{\rm Res}_{\log_\theta (c^g(P))}$ defines a global density. In Corollary~\ref{cor:resclimClifford}, we infer from the above localisation formula (\ref{eq:loctildeomegadiff}) a second {\it localisation formula} for operators in ${\rm Diff}(M,\Sigma M)$:
\begin{equation}\label{eq:locform}\omega_{\log_\theta (P)}^{{ {\rm sRes}}}(p)=(-2{\rm i})^{n/2} \widetilde\omega_{\log_\theta (\widetilde \Pp^{\lim})}^{{ {\rm Res}}}(x),
\end{equation}
where $P$ is a rescalable geometric differential operator in ${\rm Diff}(M,\Sigma M)$ of Agmon angle $\theta$ which is even for the $\Z_2$-grading $\Sigma M=\Sigma^+ M\oplus \Sigma^- M$. This formula expresses the residue density~\smash{$\omega^{\rm sRes}_{\log_\theta(P)}(p)$} at a point $p$ in terms of a local density \smash{$ \widetilde\omega_{\log_\theta (\widetilde \Pp^{\lim})}^{{ {\rm Res}}}(x)$} of the limit $\widetilde \Pp^{\lim}$ of the rescaled operators $\widetilde \Pp_\lambda^{\rm Ge}$.

 The localisation formula \eqref{eq:locform} applied to the square of the Dirac operator (see Proposition~\ref{prop:Dg2Geztlerrescalable}), which is proven to be a rescalable geometric differential operator,
confirms the results of~\cite[Sec\-tion~3.5.3.3]{Sc} (identification of (3.5.3.12) and (3.5.3.40)). Although the limit oper\-ator~$\widetilde \Pp^{\lim}$ is expected to have a simpler form than the original operator as in the case of the Dirac operator, computing \smash{$\widetilde\omega_{\log_\theta (\widetilde \Pp^{\lim})}^{{ {\rm Res}}}(x)$} nevertheless remains a challenge since it involves its $(-n)$-th homogeneous symbol.

\section[The Wodzicki residue density for classical pseudodifferential operators]{The Wodzicki residue density\\ for classical pseudodifferential operators}

 In this section, we review the definition of the Wodzicki residue for classical pseudodifferential operators acting on sections of a given vector bundle. We recall the covariance property of the Wodzicki residue under local contractions (see Proposition~\ref{prop:flambdasharpA}). We also recall how the Wodzicki residue extends to logarithms of classical pseudodifferential operators with appropriate spectral properties and refer to this extension as logarithmic Wodzicki residue. Specialising to the trivial vector bundle, we show a localisation formula for the logarithmic residue of scalar differential operators. It identifies the logarithmic residue density at the point $p$ of a differential operator~$P$ with the logarithmic residue density at any point $x$ in a small neighborhood of $p$ of the same operator localised at $p$ (see Proposition~\ref{prop:reslim}).

\subsection{The Wodzicki residue for classical pseudodifferential operators}
Let $(E,\pi,M)$ be a vector bundle over $M$, an $n$-dimensional smooth manifold, of finite rank and let~$\Psi_{\rm cl}(M, E)$ denote the algebra of classical polyhomogeneous pseudodifferential operators acting on the space $C^\infty(M, E)$ of smooth sections of~$E$. These are linear maps $Q\colon C^\infty(M, E)\to C^\infty(M, E)$, which read $Q=\sum_{i\in I} Q_{U_i} +S_Q$, where given a partition of unity $(\chi_i, i\in I)$ of~$M$ subordinated to a finite open covering $(U_i, i\in I)$ of~$M$, the operators $Q_{U_i}:= \psi_i Q \chi_i $ are localisations of $Q$ in open subsets $U_i$ of $M$, and $S_Q$ is a smoothing operator~-- it maps any Sobolev section to a smooth section. Here $(\psi_i, i\in I)$ are smooth functions compactly supported with support in $U_i$ and which are identically equal to one on the support of $\chi_i$ for any $i\in I$. Since we are interested in singular linear forms which vanish on smoothing operators, we reduce our study to localised operators $Q_U$. As these choices will not influence our results, we drop the explicit mention of the localisation and simply write $Q$. A pseudodifferential operator (localised on some open subset $U$ of $M$) acting on $C^\infty(U, E)$ is called {\it classical} or {\it polyhomogeneous} if it is a~linear combination of pseudodifferential operators $Q$ whose (local) symbol $\sigma(Q)$ -- which lies in {$C^\infty(T^*U ,\operatorname{End}(V))$},
in any local trivialisation of $E$ over $U$ -- has a polyhomogeneous expansion of the form
\[\sigma(Q)\sim\sum_{j=0}^\infty \sigma_{m-j}(Q), \]
 with
$m$ in $ \C$, the order of $Q$. Explicitly, for any $N$ in $\N$, the difference
$\sigma(Q)-\sum_{j=0}^N\chi \sigma_{m-j}(Q)$ is a smooth pseudodifferential symbol of order no larger than $\operatorname{Re}(m)-N$, with $\chi$ a smooth function which vanishes in a neighborhood of zero, and $ \sigma_{\alpha}(Q)$ positively homogeneous of degree $\alpha\in \C$, that is,
\begin{equation*}
\sigma_{\alpha}(Q) (x,\lambda \xi)=\lambda^{\alpha} \sigma_{\alpha}(Q)(x, \xi)\end{equation*}
for any { $(x,\xi)\in T^*U $} and $\lambda>0$. For further details, we refer to classical books on the subject such as \cite{Sh}, see also \cite[Example~1.1.8]{Sc}.
We also consider the class of logarithmic pseudodifferential operators, namely those whose symbols have a log-polyhomogeneous expansion of the form
\[
\sigma(Q)(x,\xi)=m \log (|\xi|){\rm Id} +\sigma_{\rm cl}( Q)(x,\xi),
\]
where $\sigma_{\rm cl}(Q)$ is a classical symbol of nonpositive order.
We define the local {\it residue density}\footnote{$s$-densities on an $n$-dimensional real vector space $V$ are functions $\mu\colon V^n\to \R$ such that $\mu(Av_1, Av_2, \dots, A v_n)=\allowbreak \vert \det (A)\vert^s \mu(v_1, \dots, v_n)$ for any linear isomorphism $A$ of $V$ and form a one-dimensional vector space $\vert \Lambda \vert^s(V)$. An $s$-density on a manifold $M$ is a section of the $s$-density bundle $\vert \Lambda\vert^s(TM)$ over $M$ whose fibre over $x$ consists of $s$-densities on the tangent space $T_xM$. On an $n$-dimensional oriented manifold $M$, $1$-densities, also called densities, can be canonically identified with the $n$-forms on $ M$.}
\begin{equation}\label{eq:Wodzres}
 \omega_Q^{\rm Res}(x):= {\rm res}(\sigma(Q)(x,\cdot)) {\rm d}x^1\wedge\dots \wedge {\rm d}x^n ,
 \end{equation}
where $ {\rm d}x^1\wedge\dots \wedge {\rm d}x^n $ is the flat volume form in local coordinates on the (oriented) $n$-dimensional manifold $M$ and
\[ {\rm res}(\sigma(Q)(x,\cdot)):=\frac{1}{(2\pi)^n} \int_{|\xi|=1}{\rm tr}^E(\sigma_{-n}(Q) (x, \xi)) {\rm d}_S\xi,\]
where ${\rm tr}^E$ stands for the fibrewise trace on $\operatorname{End}(E)$, ${\rm d}_{S}\xi$ for the standard density on the unit sphere $\mathbb{S}^{n-1}$ obtained as the interior product of the flat volume form ${\rm d}\xi^1\wedge \dots \wedge {\rm d}\xi^n$ by the radial vector field ${\mathcal R} :=\sum_{i=1}^n \xi^i\frac{\partial}{\partial\xi^i}$, namely
\[
{\rm d}_{S}\xi:={\mathcal R}\lrcorner\bigl({\rm d}\xi^1\wedge\dots \wedge \dots \wedge {\rm d}\xi^n\bigr)= \sum_{j=1}^{n} (-1)^{j-1} \xi^j {\rm d}\xi^1\wedge\dots \wedge \widehat{{\rm d}\xi^j}\wedge \dots \wedge {\rm d}\xi^n.
\]
A priori, $\omega_Q^{\rm Res}(x),$ which is defined using a localisation of the operator $Q$ around $x$, depends on the choice of local coordinates in a neighborhood of~$x$.
M.~Wodzicki \cite{W} showed that it actually defines a global $n$-form, which can be integrated to define the linear form ${\rm Res}$ on $\Psi_{\rm cl}(M, E)$, called the Wodzicki or noncommutative residue:
\[{\rm Res}(Q):= \int_M \omega_Q^{\rm Res}(x).\]

\begin{rk}\quad
\begin{enumerate}\itemsep=0pt
\item If $(M, g)$ is an $n$-dimensional smooth manifold Riemannian manifold, we can equivalently define
\[
{\rm res}(\sigma(Q)(x,\cdot)):=\frac{1}{(2\pi)^n} \int_{S_x^*M}{\rm tr}^E(\sigma_{-n}(Q) (x, \xi)) \nu_x(\xi),
\]
as an integral over the cotangent unit sphere $S_x^* M:=\{\xi\in T_x^*M, \vert \xi\vert=1\}$ endowed with the induced Riemannian volume form $\nu_x$.
\item The Wodzicki residue easily extends to a $\Z_2$-graded vector bundle $E=E^+\oplus E^-$ replacing the fibrewise trace ${\rm tr}^E$ by a graded trace ${\rm str}^E:= {\rm tr}^{E^+}- {\rm tr}^{E^-}$, in which case we set
\[
{\rm sres}(\sigma(Q)(x,\cdot)):=\frac{1}{(2\pi)^n} \int_{|\xi|=1}{\rm str}^E(\sigma_{-n}(Q) (x, \xi)) {\rm d}_S\xi,
\]
and
\[
\omega_Q^{\rm sRes}(x):= {\rm sres}(\sigma(Q)(x,\cdot)) {\rm d}x^1\wedge\dots \wedge {\rm d}x^n.
\]
\end{enumerate}\end{rk}

\subsection{Local contractions}
Throughout the paper, $(M, g)$ denotes a smooth Riemannian manifold of dimension~$n$ and $p$ a point in $M$. The local identification uses the exponential map
\begin{equation}\label{eq:expp}
\exp_{p}\colon \ T_{p}M\supset B_r\longrightarrow U_{p}\subset M
\end{equation}
around $p$ which yields a~local diffeomorphism from a ball $B_r$ of radius $r>0$ centered at $0$ to a~local geodesic neighborhood $U_p$ of $p$.
This exponential map is combined with a rescaling leading to the map (this is the map $\exp\circ T_\epsilon$ in \cite[formula (4.4.7)]{F})
\begin{equation*}
\exp_{p} \circ h_\lambda\colon \ T_{p}M\supset B_{r/\lambda}\longrightarrow U_{p}\subset M,
\end{equation*}
where
\begin{align*}
h_\lambda\colon \ T_{p} M\supset B_{r/\lambda} &\longrightarrow B_r\subset T_{p}M, \\
	{\mathbf x} & \longmapsto \lambda {\mathbf x}.
\end{align*}

In the sequel, we use the following notations. From a given orthonormal basis $e_1(p), \dots,\allowbreak e_n(p)$ of $T_pM$ at $p\in M$, we build:
\begin{itemize}\itemsep=0pt
	\item normal geodesic coordinates at any point $x\in U_p$, by means of the map $ {B}_r\subset\R^n\to U_p$; $\bigl(x^1,\dots,x^n\bigr)\mapsto x$
\begin{equation}\label{eq:geodesiccoord}
x=\exp_p\Bigg(\sum_{i=1}^n x^i e_i(p)\Bigg)\in {\rm exp}_p(B_r)
\end{equation}
	 defined via the local exponential map $\exp_p$ in (\ref{eq:expp});
	\item a local orthonormal frame
\begin{equation}\label{eq:Op}
O_p(x, g):=\{e_1(x, g),\dots,e_n(x, g)\},\qquad x\in U_p
\end{equation}
of $T_xM$ by the parallel transport $\tau_{c}\colon T_pM\to T_xM$ along the geodesic $c(t)= \exp_p(t {\bf x}), $ with $ {\bf x}$ in $B_r\subset T_pM,$ which takes $p$ to $x= c(1)$ so that $e_j(x, g)=\tau_{c} (e_j(p))$.
\end{itemize}

{\it Unless specified otherwise, we use normal geodesic coordinates. As usual, we identify any point $x\in U_p$ with its coordinates $X:=\bigl(x^1,\dots, x^n\bigr)$.} Let $1>\lambda>0$. By means of the map $h_\lambda$,
 we define a rescaled coordinate system
\begin{equation*}
	Y:=\bigl(y^1:= \lambda x^1, \dots, y^n:= \lambda x^n\bigr)
\end{equation*}
at any point in $U_p$. Since
we have the inclusion $ {B_r \subset B_{r/\lambda}}$, the map $h_\lambda$
 induces a diffeomorphism
 \begin{align}
 {\mathfrak f}_{p,\lambda}\colon \ U_{p} &\longrightarrow U_{p}^\lambda:=\exp_p(B_{\lambda r})\subset U_p,\nonumber\\
	 \exp_{p}({\mathbf x})&\longmapsto \exp_p\circ h_\lambda( {\mathbf x})= \exp_{p}( \lambda {\mathbf x}),\label{eq:flambdaUplambda}
\end{align}
which we shall denote by ${\mathfrak f}_\lambda$ for simplicity. As a consequence of the above constructions, we have
\begin{gather}
\mathfrak{f}_\lambda^*\biggl(\frac{\partial}{\partial x^{i}}\biggr)=\frac{\partial}{\partial y^i}=\lambda^{-1} \frac{\partial}{\partial x^{i}}\circ \mathfrak{f}_\lambda \qquad\text{and}\qquad \mathfrak{f}_\lambda^* {\rm d}x^i={\rm d}y^i=\lambda {\rm d}x^i\circ \mathfrak{f}_\lambda.\label{eq:derivativedifferent}
\end{gather}

\subsection{The behaviour of the Wodzicki residue under local contractions}
Let us now recall the general fact on pull-back of operators. Any local diffeomorphism $\mathfrak{f}\colon U\to V$ induces a local transformation on a localised pseudodifferential {operator} as follows: Given any~$Q$ in $\Psi_{\rm cl}(V,E)$, where $(E,\pi,M)$ is a vector bundle over $M$, we define $\mathfrak{f}^\sharp Q\in \Psi_{\rm cl}(U,\mathfrak{f}^*E)$ by
\begin{equation}\label{eq:flambda}
 \bigl(\mathfrak{f}^\sharp Q\bigr) s:= \mathfrak{f}^* \bigl( Q\bigl(\bigl( \mathfrak{f}^*\bigr)^{-1}(s)\bigr)\bigr)= Q\bigl(s\circ \mathfrak{f}^{-1} \bigr)\circ \mathfrak{f},
 \end{equation}
where $s$ is any local section in $\mathfrak{f}^*E$ above $U$. Here, $\mathfrak{f}^*E$ is the pull-back bundle over $U$ of the bundle $E$ given by
\[\mathfrak{f}^*E=\{(x,y)\in U\times E \mid \mathfrak{f}(x)=\pi(y)\}.\]
The following lemma is an easy consequence of the transformation property of symbols under the local diffeomorphism ${\mathfrak f}_{ \lambda}$. We nevertheless provide an explicit proof.

\begin{lem}[{compare with~\cite[p.~381]{Sc}}]
Given any $Q$ in $\Psi_{\rm cl}(M,E)$,
 we have for small enough positive $\lambda$, 	
\begin{equation}\label{eq:sigmaflambda}
\sigma\bigl( {\mathfrak f}_{ \lambda}^\sharp Q\bigr)(x,\xi)
= 	\sigma(Q)\bigl( 	{\mathfrak f}_{ \lambda}(x),\bigl( \bigl( ({\mathfrak f}_{ \lambda})_*\bigr)^t\bigr)^{-1} (\xi)\bigr)
=\sigma ({ Q})\bigl({ { {\mathfrak f}}_{ \lambda}(x)}, \bigl({\mathfrak f}_{ \lambda}^*\bigr)^t (\xi)\bigr),
\end{equation}
at any given point $x$ in $U_p$.
\end{lem}
\begin{proof} A local diffeomorphism $\mathfrak{f}\colon U\to V$ between two open subsets $U$ and $V$ of $M$, induces a map $ ({\mathfrak f}^*)^t_x\colon T_x^*U\to T_{\mathfrak{f}(x)}^*V$ and the symbol $\sigma(Q)$ of $Q$ transforms as (see, e.g., \cite[equation~(4.2.1)]{Sh})
\begin{equation}\label{eq:sigmafsharp}
\sigma\bigl({\mathfrak{f}}^\sharp Q\bigr)(x, \xi)= \sigma(Q)\bigl(\mathfrak{f}(x), \bigl( ({\mathfrak f}_{ *})^t\bigr)^{-1}(\xi)\bigr)+ \text{lower order terms}.
\end{equation}
Here ``lower order terms'' stands for the push forward by $\mathfrak{f}$ of the sum  $\sum_{\vert \alpha\vert >0}
\frac{1}{\alpha!} \Phi_\alpha(x, \xi)
\partial_\xi^\alpha \sigma\left(x, \xi \right) $ with
\[
\Phi_\alpha (x, \xi):=
 D_z^\alpha {\rm e}^{{\rm i}\left\langle \varphi_x^\mathfrak{f}(z), \xi\right\rangle }\big\vert_{z=x} \qquad {\rm and}\qquad \varphi_x^\mathfrak{f}(z):=\mathfrak{f}(z)-\mathfrak{f}(x)-{\rm d}\mathfrak{f}(x)(z-x),
 \]
 which is a polynomial in $\xi$ of degree
 $\leq \frac{\vert \alpha\vert}{2}$
 whose coefficients are linear combinations of products of derivatives
 $ \prod_\gamma \partial^\gamma {\mathfrak f}(x)$ of $\mathfrak{f}$ at $x$ with $\sum\vert \gamma\vert=\vert\alpha\vert$ and $\vert\gamma\vert\geq 2$.
 For {$\mathfrak{f}=\mathfrak f_\lambda$ with $U=U_p$ and $V=U_p^\lambda$}, we have $\partial_i\partial_j \mathfrak{f}\equiv 0$ for any indices $i$, $j$ running from $1$ to $n$, so that the lower order terms vanish leading to (\ref{eq:sigmaflambda}).
 \end{proof}

\begin{prop} \label{prop:flambdasharpA} Let $E$ be a vector bundle over $M$ of finite rank. For any given $Q$ in $\Psi_{\rm cl}(M,E)$, the $n$-form $\omega^{\rm Res}_Q $ transforms covariantly under contractions ${\mathfrak f}_\lambda$ as
\begin{equation*}
\omega_Q^{\rm Res}\circ{\mathfrak f}_\lambda =\omega^{\rm Res}_{{\mathfrak f}_\lambda^\sharp Q},
\end{equation*}
for any $\lambda>0$ small enough.
\end{prop}
 \begin{proof} Applying the local residue density \eqref{eq:Wodzres} at the point $\mathfrak{f}_\lambda(x)$ {with $x\in U_p$}, we have
\begin{align*} 	
(2\pi)^n\omega_Q^{{ {\rm Res}}}({\mathfrak f}_{ \lambda}(x))
&= \biggl( \int_{\vert \xi\vert=1} {\rm tr}^E(\sigma_{ {-n}}({ Q})({ { {\mathfrak f}}_{ \lambda}(x)}, \xi)) {\rm d}_S\xi\biggr) {\rm d}{ y^1}\wedge \dots \wedge {\rm d}{y^n} \\
&= \biggl( \int_{\vert \xi\vert=1}{\rm tr}^E\bigl( \sigma_{ {-n}}({Q})\bigl({ { {\mathfrak f}}_{ \lambda}(x)},{ \lambda}^{ -1} \xi\bigr)\bigr) {\rm d}_S\xi\biggr) {\rm d}{ x}^1\wedge \dots \wedge {\rm d}{x}^n \\
&= \biggl( \int_{\vert \xi\vert=1} {\rm tr}^E\bigl(\sigma_{ {-n}}({ Q})\bigl({ { {\mathfrak f}}_{ \lambda}(x)}, ({\mathfrak f}_{ \lambda}^*)^t (\xi)\bigr)\bigr) {\rm d}_S\xi\biggr) {\rm d}{ x}^1\wedge \dots \wedge {\rm d}{x}^n \\
&\bui{=}{\eqref{eq:sigmaflambda}} (2\pi)^n \omega^{ {\rm Res}}_{{ {\mathfrak f}}_{ \lambda}^\sharp (Q)}(x).
\end{align*}
This finishes the proof of the lemma.
\end{proof}

\subsection{Logarithmic residue density} \label{sec:logoperator}
For later purposes, we review here how the Wodzicki residue can be extended to logarithms of pseudodifferential operators as defined in Appendix~\ref{appendix2}.
As before, we consider a vector bundle~$E$ over~$M$ of finite rank. We say that an operator $Q$ in $\Psi_{\rm cl}(M, E)$ of positive real\footnote{The order is assumed to be real so as to ensure that does not intersect all rays. } order~$m$ has a~principal angle $\theta\in [0,2\pi)$ (see \cite[Section~1.5.7.1]{Sc}) if the leading symbol matrix $\sigma_L(Q) (x, \xi):=\sigma_m(Q)(x,\xi)$ has no eigenvalue on the ray $L_\theta:= \{ r {\rm e}^{{\rm i}\theta} ,\, r\geq 0\}$ for every $(x, \xi)\in T^*U\setminus U\times \{0\}$.
	 In particular, the operator is elliptic and, therefore, has a purely discrete spectrum. A principal angle $\theta$ of an operator $Q$ is said to be an Agmon angle\footnote{One can actually build an Agmon angle from a small perturbation of any principal angle.} if there exists a solid angle of the ray
\[
\Lambda_{\varepsilon,\theta}=\big\{r{\rm e}^{{\rm i}\alpha},\, r\geq 0, \,\theta-\varepsilon\leq \alpha\leq \theta+\varepsilon\big\},
\]
	 for some $\varepsilon>0$, that contains no eigenvalue of $Q$.
	 In that case, the operator $Q$ is invertible.

\begin{rk} We shall drop the explicit mention of the principal angle when we can choose $\theta=\pi$.\end{rk}
For such an operator $Q$, we can define the complex power $Q_\theta^z$ for $z\in \C$ and the logarithm~$\log_\theta(Q)$ as in Appendix~\ref{appendix2}. It is ``nearly'' classical in so far as its local symbol differs from a~classical symbol by a logarithm term. Indeed, it is shown in \cite[formula~(2.6.1.11)]{Sc} that the symbol of the logarithm reads
\begin{equation}
\label{eq:sigmmalog}
\sigma(\log _\theta (Q))(x,\xi)=m \log (|\xi|){\rm Id} +\sigma_{\rm cl}(\log _\theta (Q))(x,\xi),
\end{equation}
where $\sigma_{\rm cl}(\log _\theta(Q))$ is a classical symbol of order zero with homogeneous components \linebreak $\sigma_{-j}(\log _\theta(Q))$ of degree $-j, j\geq 0$ given by (this follows from the formula above (2.6.1.11) on p.~219 in \cite{Sc})
\begin{equation*}
\sigma_{-j}(\log _\theta (Q))(x,\xi)=|\xi|^{-j}\biggl(\partial_z\biggl(\sigma_{mz-j}(Q^z_\theta)\biggl(x,\frac{\xi}{|\xi|}\biggr)\biggr)\biggr)_{z=0}.
\end{equation*}

The fact that the logarithmic part of the symbol vanishes on the cotangent unit sphere underlies the extendibility of the Wodzicki residue to logarithmic pseudodifferential operators (for a detailed discussion, we refer the reader to \cite[Section~2.7.1]{Sc}). In analogy with (\ref{eq:Wodzres}), we~set
\begin{gather}
{\rm res}(\sigma(\log_\theta (Q))(x,\cdot)) := \frac{1}{(2\pi)^n} \int_{S_x^*M}{\rm tr}^E(\sigma_{-n}(\log_\theta (Q)) (x, \xi)) {\rm d}_S\xi, \nonumber\\
\omega_{\log _\theta (Q)}^{\rm Res}(x) := {\rm res}(\sigma(\log_\theta (Q))(x,\cdot)) {\rm d}x^1\wedge \dots \wedge {\rm d}x^n,\label{eq:reslog}
\end{gather}
which we call the {\it logarithmic residue density of $P$}. Given a local diffeomorphism $\mathfrak{f}\colon U\to V$ and an operator $Q\in\Psi_{\rm cl}(M, E)$ with Agmon angle $\theta$, the operator $\mathfrak{f}^\sharp Q$ defined in \eqref{eq:flambda} lies in $\Psi_{\rm cl}(M, \mathfrak{f}^*E)$ with the same Agmon angle $\theta$, since equation \eqref{eq:sigmafsharp} implies that
\[
\sigma_L\bigl({\mathfrak{f}}^\sharp Q\bigr)(x, \xi)= \sigma_L(Q)\bigl(\mathfrak{f}(x), (\mathfrak{f}^*)^t (\xi)\bigr).
\]
Furthermore, the relation $\mathfrak{f}^\sharp(Q-\lambda)^{-1}=\bigl(\mathfrak{f}^\sharp Q-\lambda\bigr)^{-1}$ gives that
{\begin{equation}\label{eq:fQz}
\mathfrak{f}^\sharp(Q_\theta^z)= \bigl(\mathfrak{f}^\sharp Q_\theta\bigr)^z
\end{equation}
for any complex number $z$ with negative real part. Since $\mathfrak{f}^\sharp(Q^k)= \bigl(\mathfrak{f}^\sharp Q\bigr)^k$ for any positive integer~$k$, it follows from the construction of the extension $Q_\theta^z$ to any complex number $z$, that Property~(\ref{eq:fQz}) extends to $z\in \C$. Similarly, one shows that
\begin{equation*}
\mathfrak{f}^\sharp(\log _\theta (Q))=\log _\theta\bigl( \mathfrak{f}^\sharp Q\bigr),
\end{equation*}
in other words, $\mathfrak{f}^\sharp$ commutes with the functional calculus.} On the grounds of formula (\ref{eq:sigmmalog}), $\sigma_{-n}(\log _\theta (Q))) = (\sigma_{\rm cl})_{-n}(\log _\theta (Q)))$ so that one can easily adapt the proof of Proposition~\ref{prop:flambdasharpA} to show the covariance of the logarithmic residue:
\begin{equation}\label{eq:fReslog}
\omega_{\log_\theta (Q)}^{\rm Res}\circ{\mathfrak f}_\lambda =\omega^{\rm Res}_{ \log_\theta ({\mathfrak f}_\lambda^\sharp Q)},
\end{equation}
where we have used the fact that $\mathfrak{f}_\lambda^\sharp$ and $\log _\theta$ commute.
\subsection{A localisation formula for the logarithmic residue density} We now focus on logarithms of scalar differential operators, for which we prove a localisation formula for the Wodzicki residue density. In the sequel, we use the following notations. For any multiindex $\gamma=\{i_1,\dots,i_s\}$, we set
\begin{equation} \label{eq:dgammax}
D_X^\gamma:=\frac{\partial}{\partial x^{i_1}}\cdots\frac{\partial}{\partial x^{i_s}},
\end{equation}
in the local normal geodesic coordinates $X=\bigl(x^1,\dots,x^n\bigr)$ at point $x$ with the usual identification $x\leftrightarrow X$. {\it To simplify notations, unless this gives rise to an ambiguity, we henceforth write $D^\gamma$ instead of $D_X^\gamma$.}
	
Given a vector bundle $E\to M$ of rank $k$, trivialised over an open subset $U$ of $M$, $D^\gamma$ acts on a local section $s|_U=\sum_{i=1}^k \alpha_i s_i\big|_U$ on $U$ by
\begin{equation}\label{eq:Ds}
D^\gamma s:=\sum_{i=1}^k D^\gamma(\alpha_i) s_i.
\end{equation}
Here $\{s_i\}_{i=1,\dots,k}$ is a basis of the bundle $E|_U$ in the local trivialisation $E|_U\simeq U\times \R^k$. A~differential operator of order $m\in \Z_{\geq 0}$ reads $P= \sum_{\vert \gamma\vert \leq m} P_\gamma D^\gamma$, which means that in the local trivialisation $E\vert_U\simeq U\times \R^k$ of $E$, it acts as
 \begin{equation}\label{eq:PgLambdaT*Mcoord}
 P\Bigg(\sum_{j=1}^k \alpha_j s_j\Bigg)=\sum_{\vert \gamma\vert \leq m}\sum_{i, j=1}^k (P_\gamma)_{ij} D^\gamma (\alpha_j) s_i,
 \end{equation}
 where we have used equation~\eqref{eq:Ds}. Differential operators form an algebra ${\rm Diff}(M, E)$ and we have the following isomorphism of $C^\infty$-modules:
\begin{equation*}
{\rm Diff} (M, E)\simeq {\rm Diff} (M) \otimes_{C^\infty(M)} C^\infty(M, \operatorname{End}(E)),\end{equation*}
where we have set ${\rm Diff}(M):= {\rm Diff}(M, M\times \R)$.
Following \cite{vEY}, we define a family of rescaled differential operators for any $P\in {\rm Diff}(M, M\times\mathbb{R})$ by
\begin{equation}\label{eq:PPlambda}
\Pp :=\lambda^m P, \qquad \lambda>0,
\end{equation}
and set for any small positive $\lambda$
\begin{equation}\label{eq:PPlambdaflambda}
\widetilde \Pp_\lambda:=\lambda^m\mathfrak{f}_\lambda^\sharp P.
\end{equation}

In local normal geodesic coordinates, we have ${\mathfrak f}_\lambda^\sharp D^\gamma= \lambda^{-\vert \gamma\vert} D^\gamma$ so that the family of rescaled operators built from a differential operator $P =\sum_{\vert\gamma\vert\leq m}P_\gamma D^\gamma$,
 locally reads (these and the above notations $\Pp_\lambda$ are borrowed from \cite{vEY})
\[
\widetilde \Pp_\lambda=\sum_{ |\gamma|\leq m } \lambda^{m-|\gamma|} (P_{\gamma}\circ \mathfrak{f}_\lambda) D^{\gamma}.
\]
As $\lambda$ tends to zero, $\widetilde \Pp_\lambda$ converges to the operator $P$ evaluated at the limit point $p$
\begin{equation}\label{eq:limflambdaPlambda}
\underset{\lambda \to 0}{\lim} \widetilde \Pp_\lambda\big\vert_{U_p }=\sum_{\vert \gamma\vert=m}P_\gamma(p) D^\gamma\vert_p=P|_p,
\end{equation}
where $P_{\gamma}(p)$ corresponds to $P_{\gamma}(x)$ evaluated at the reference point $p$.
In the following, we state a~localisation formula for the residue of the logarithm of a differential operator (see \cite[formula~(3.5.3.33), p.~382]{Sc} for a similar formula).

\begin{prop}\label{prop:reslim} For any differential operator $P$ in $ {\rm Diff}(M, M\times\mathbb{R})$ with Agmon angle $\theta$, we have the following localisation formula:
	\[	 \omega_{\log_\theta (P)}^{{ {\rm Res}}}(p)= \omega_{\log_\theta (P|_p)}^{{ {\rm Res}}}(x)\]
for all $x\in U_p$.
\end{prop}

\begin{proof} We first observe that for small positive $\lambda$
\[
\log_\theta \bigl(\widetilde \Pp_\lambda\bigr)= \log_\theta \bigl(\lambda^m \mathfrak{f}_\lambda^\sharp P\bigr) =(m \log\lambda) {\rm Id}+ \log_\theta\bigl({\mathfrak f}_\lambda^\sharp P\bigr),
\]
where $m$ is the order of $P$. Since the residue density vanishes on differential operators and hence on ${\rm Id}$, we have \smash{$\omega^{\rm Res}_{\log_\theta (\widetilde \Pp_\lambda)} =\omega^{\rm Res}_{ \log _\theta({\mathfrak f}_\lambda^\sharp P)}$}. Equation~(\ref{eq:fReslog}) implies that $\omega^{\rm Res}_{\log_\theta (\widetilde \Pp_\lambda)}=\omega^{\rm Res}_{\log_\theta (P)}\circ \mathfrak{f}_\lambda$. We then take the limit as $\lambda\to 0$, by which $\widetilde \Pp_\lambda$ tends to $P\vert_p$ by \eqref{eq:limflambdaPlambda}. The continuity of the logarithm combined with the continuity of the Wodzicki residue for the Fr\'echet topology of (log-)classical operators of constant order then yields the statement of the proposition.
\end{proof}

\section[A local Berezin type n-form on Psi\_\{cl\} (M, Lambda T\^{}*M)]{A local Berezin type $\boldsymbol{n}$-form on $\boldsymbol{\Psi_{\rm cl} (M, \Lambda T^*M)}$}

In this section, we define a local $n$-form $\widetilde \omega^{\rm Res}$ on $\Psi_{\rm cl} (M, \Lambda T^*M)$ (see equation~\eqref{eq:omegatildeRes}), which unlike the Wodzicki density, is not covariant under contractions defined in the previous section. We give in Proposition~\ref{prop:pullbacklocalform} the behaviour of this local $n$-form $\widetilde \omega^{\rm Res}$ under Getzler rescaling map
 (see Definition~\ref{getzlerrescaling}) combined with the local contractions. When the manifold $M$ is spin and for a differential operator $P$ acting on smooth sections of its spinor bundle, we use the expression of the super trace in terms of a Berezin integral (see \eqref{eq:supertraceformula}) to relate the local $n$-form $\widetilde \omega_{\log_\theta (c^g(P))}^{\rm Res}$ of the logarithm (with spectral cut $\theta$) of $c^g(P)$ (defined in equation~\eqref{eq:cgp}) to its (super-)Wodzicki residue $ \omega_{\log_\theta (c^g(P))}^{\rm sRes}$. Much of this section is inspired from Simon Scott's approach to the local Atiyah--Singer index theorem by means of the Wodzicki residue \cite[Section~3.5.3]{Sc}.

\subsection{The Getzler rescaling map}

To simplify the notation, we set $ \otimes_r^q V:=V^{\otimes q}\otimes (V^*)^{\otimes r}$.
\begin{defn} \label{getzlerrescaling}
The Getzler rescaling map is the tensor bundle morphism defined for any $\lambda>0$, by
\begin{align}
U_\lambda\colon \otimes_r^q V &\longrightarrow \otimes_r^q V, \nonumber\\
\hphantom{U_\lambda\colon}{} \ t &\longmapsto 	 \lambda^{ q-r} t.\label{eq:getzlerrescaling}
\end{align}
\end{defn}
The Getzler rescaling map $U_\lambda$ restricted to $\Lambda V$ induces a map
\begin{align}
U_\lambda^\sharp\colon \ \operatorname{End}(\Lambda V) & \longrightarrow \operatorname{End}(\Lambda V), \nonumber\\
Q \longmapsto U_\lambda^\sharp Q\colon \ \omega & \mapsto U_\lambda Q U_\lambda^{-1}\omega,\label{eq:actionulambda}
\end{align}
which satisfies
\begin{equation}\label{eq:ulambdaformes}
U_\lambda^\sharp(v\wedge\bullet)= \lambda^{-1}v\wedge\bullet \qquad\text{and}\qquad U_\lambda^{\sharp}\bigl(v^{\sharp_g}\lrcorner\bullet\bigr)= \lambda v^{\sharp_g}\lrcorner\bullet.
\end{equation}
for any $v\in V^*$. Combining (\ref{eq:ulambdaformes}) with the Clifford map $ c^g\colon \Cl(V) \longrightarrow \operatorname{End}(\Lambda V)$ defined in (\ref{eq:cliffend}) Appendix~\ref{sec:appendix3}, on the covector $v\in V^*$ by $ c^g(v)\bullet =v\wedge\bullet -v^{\sharp_g} \lrcorner\bullet$
 yields the map
 \begin{equation*}
 U_\lambda^\sharp\circ c^g\colon \ \Cl(V) \longrightarrow \operatorname{End}(\Lambda V)
 \end{equation*}
 given by
	\[ \bigl(U_\lambda^\sharp\circ c^g\bigr)(v)\bullet=\lambda^{-1} v\wedge\bullet-\lambda v^{\sharp_g}\lrcorner\bullet.\]
We have the following straightforward lemma that we will use later.

\begin{lem} Let ${\rm e}^I:={\rm e}^{i_1}\cdot_g {\rm e}^{i_2}\dots \cdot_g {\rm e}^{i_k}$ for $i_1<i_2<\dots<i_k$ with {$|I|= k$}, it follows that
\begin{equation}\label{eq:limUlambdac}
\underset{\lambda\to 0}{\lim} \lambda^{|I|}\bigl(U_\lambda^\sharp\circ c^g\bigr)({\rm e}^I)= {\rm e}^I\wedge,
\end{equation}
 where ${\rm e}^I\wedge:={\rm e}^{i_1}\wedge\dots \wedge {\rm e}^{i_k}$.
 \end{lem}

\subsection[A local n-form on Psi\_\{cl\} (M, Lambda T\^{}*M) and Getzler rescaling]{A local $\boldsymbol{n}$-form on $\boldsymbol{\Psi_{\rm cl} (M, \Lambda T^*M)}$ and Getzler rescaling}
In order to define the local $n$-form, we fix a normal geodesic neighborhood $U_p$ around a point~$p$ in $M$. For $(x, \xi)\in T^*U_p$, we consider the symbol $\sigma(Q)(x, \xi) \in \operatorname{End}(\Lambda T_x^*U_p)$ of an operator $Q$ in~$\Psi_{\rm cl}(M, \Lambda T^*M)$ in the corresponding coordinate chart. Its homogeneous component $\sigma_{-n}(Q)(x, \xi)$ of degree $-n$ evaluated in ${\bf 1}_x$ yields a differential form
$\sigma_{-n}(Q)(x, \xi) {\bf 1}_x \in \Lambda T_x^*U_p$. Hence we define
\begin{equation}\label{eq:omegatildeRes}
\widetilde \omega^{\rm Res}_Q(x):= \frac{1}{(2\pi)^n} \int_{|\xi|=1} [ \sigma_{-n} (Q) (x, \xi) {\bf 1}_x ]_{[n]} {\rm d}_S\xi,
\end{equation}
where $\alpha_{[n]}$ stands for the part of degree $n$ of a form $\alpha$ in $\Lambda T_x^*M$.

\begin{rk}\quad
\begin{itemize}\itemsep=0pt
\item Note that this differs from the Wodzicki residue density. Contrarily to \smash{$\omega^{\rm Res}_Q$} which is covariant with respect to the action of ${\mathfrak f}_{ \lambda}$, as we shall see shortly, $\widetilde \omega^{\rm Res}_Q$ is not. Getzler's rescaling map will enable us to compensate this lack of covariance.
\item The above constructions generalise beyond classical pseudodifferential operators, to logarithmic pseudodifferential operators. For a differential operator $P$ in ${\rm Diff}(M, \Lambda T^* M)$ with Agmon angle $\theta$, similarly to (\ref{eq:omegatildeRes}), we define
\begin{equation*}
\widetilde \omega^{\rm Res}_{\log_\theta(P)}(x):= \frac{1}{(2\pi)^n} \int_{|\xi|=1}[ \sigma_{-n} (\log_\theta (P)) (x, \xi) {\bf 1}_x]_{[n]} {\rm d}_S\xi.
\end{equation*}
\end{itemize}
\end{rk}

The maps $U_\lambda^\sharp$ defined in \eqref{eq:actionulambda} induce a transformation on differential operators as follows: for any $P =\sum_{\vert \gamma\vert \leq m} P_\gamma D^\gamma$ in ${\rm Diff}(M,\Lambda T^*M)$ of order $m$, we define
\begin{equation}\label{eq:ulambdap}
U_\lambda^\sharp P:=\sum_{|\gamma|\leq m} U_\lambda^\sharp( P_\gamma) D^\gamma \in {\rm Diff}(M, \Lambda T^*M).
\end{equation}
A first direct consequence of \eqref{eq:ulambdap} is that
$\sigma_L\bigl(U_\lambda^\sharp P\bigr)= U_\lambda^\sharp\sigma_L(P) $ so that the operator $U_\lambda^\sharp P$ is also of order $m$ and has Agmon angle $\theta$. We build $\bigl(U_\lambda^\sharp P\bigr)_\theta^z$ and $\log_\theta\bigl(U_\lambda^\sharp P\bigr)$ following the construction in Section~\ref{sec:logoperator}.

\begin{lem} For any differential operator $P\in {\rm Diff}(M,\Lambda T^*M)$ with Agmon angle $\theta$ and order~$m$, we have
	\begin{gather}
		\sigma_{mz-j}\bigl(\bigl(U_\lambda^\sharp P\bigr)_\theta^z\bigr)(x,\xi) = U_\lambda^\sharp (\sigma_{mz-j}(P_\theta^z)(x,\xi) ),\nonumber \\
		\sigma_{-j}\bigl( \log_\theta\bigl(U_\lambda^\sharp P\bigr)\bigr)(x,\xi) = U_\lambda^\sharp (\sigma_{-j}(\log_\theta(P))(x,\xi) )\label{eq:ulambdasharpdef}
	\end{gather}
	for any $(x,\xi)\in T^*U$ and $j\geq 0$.
\end{lem}
\begin{proof} We prove the first identity, the second one can be shown in a similar manner. Let~$\mu$ lie on the contour $\Gamma_\theta$. Since
		$ U_\lambda^\sharp P-\mu= U_\lambda^\sharp (P-\mu) $, the same property holds on the symbolic level 	$ \sigma\bigl(U_\lambda^\sharp P\bigr)-\mu = U_\lambda^\sharp\bigl( \sigma (P)-\mu \bigr)$. As a result, the product formula $\sigma\bigl(\bigl( U_\lambda^\sharp P-\mu\bigr)^{-1}\bigr)\star \bigl(\sigma\bigl(U_\lambda^\sharp P\bigr)-\mu \bigr)={\rm Id}$ (see, e.g., \cite[equation~(4.8.2.2)]{Sc}) which determines the homogeneous components $\sigma_{-m-j}\bigl(\bigl( U_\lambda^\sharp P-\mu\bigr)^{-1}\bigr)$ of the resolvent with $j\in \Z_{\geq 0}$, reads $\sigma\bigl(\bigl( U_\lambda^\sharp P-\mu\bigr)^{-1}\bigr)\star \bigl(U_\lambda^\sharp(\sigma( P)-\mu 1)\bigr)={\rm Id}$. It follows that $\sigma_{-m-j}\bigl(\bigl( U_\lambda^\sharp P-\mu\bigr)^{-1}\bigr)
		=U_\lambda^\sharp\sigma_{-m-j}\left(( P-\mu)^{-1}\right)$ for $j\in \Z_{\geq 0}$. { Using equation~\eqref{eq:symbolqz} in Appendix~\ref{appendix2}}, this yields for $\operatorname{Re}(z)<0$ and $j\in \Z_{\geq 0}$
\begin{align*}
\sigma_{mz-j}( \bigl(U_\lambda^\sharp P\bigr)_\theta^z)(x,\xi)&{}= \frac{{\rm i}}{2\pi} \int_{\Gamma_\theta} \mu_\theta^z \sigma_{-m-j}\bigl(\bigl(U_\lambda^\sharp P-\mu\bigr)^{-1}\bigr)(x,\xi) {\rm d}\mu\\
&{}= \frac{{\rm i}}{2\pi} \int_{\Gamma_\theta} \mu_\theta^z U_\lambda^\sharp\sigma_{-m-j}\bigl((P-\mu)^{-1}\bigr)(x,\xi) {\rm d}\mu\\
&{}= U_\lambda^\sharp(\sigma_{mz-j}(P_\theta^{z})(x,\xi)).
\end{align*}
These identities can then be extended to any complex number $z$. For $\operatorname{Re}(z)<k$ with $k\in \N$, we write $\bigl(U_\lambda^\sharp P\bigr)_\theta^z= \bigl(U_\lambda^\sharp P\bigr)^k \bigl(U_\lambda^\sharp P\bigr)_\theta^{z-k}$. Since $ \sigma_{mk-j} \bigl(\bigl(U_\lambda^\sharp P\bigr)^k\bigr)= \sigma_{mk-j} \bigl(U_\lambda^\sharp P^k\bigr)= U_\lambda^\sharp\sigma_{mk-j} ( P^k)$ for any $j\in \Z_{\geq 0}$, it follows from (\ref{sigmajzk}) that
\begin{align*}
\sigma_{mz-j}\bigl(\bigl(U_\lambda^\sharp P\bigr)_\theta^z\bigr)&{}= \sum_{a+b+\vert \alpha\vert=j} \frac{(-{\rm i})^{\vert \alpha \vert}}{\alpha!} \partial_\xi^\alpha \sigma_{mk-a} \bigl(\bigl(U_\lambda^\sharp P\bigr)^k\bigr) \partial_x^\alpha\sigma_{m(z-k)-b}\bigl( \bigl(U_\lambda^\sharp P\bigr)_\theta^{z-k}\bigr)\quad \forall j\in \Z_{\geq 0}\\
&{}= \sum_{a+b+\vert \alpha\vert=j} \frac{(-{\rm i})^{\vert \alpha \vert}}{\alpha!} U_\lambda^\sharp\partial_\xi^\alpha \sigma_{mk-a} \bigl(P^k\bigr) U_\lambda^\sharp\partial_x^\alpha\sigma_{m(z-k)-b}\bigl( P_\theta^{z-k}\bigr)\quad \forall j\in \Z_{\geq 0}\\
&{}= U_\lambda^\sharp \sigma_{mz-j}( P_\theta^z).
\tag*{\qed}
\end{align*}
\renewcommand{\qed}{}
\end{proof}
\begin{lem} 
Given any $P\in {\rm Diff}(M,\Lambda T^*M)$ we have for any $\lambda>0$
\begin{equation}\label{eq:flambdaUlambda}
\bigl(U_\lambda^\sharp\circ\mathfrak{f}^\sharp\bigr)(P)=\bigl(\mathfrak{f}^\sharp\circ U_\lambda^\sharp\bigr)(P)
\end{equation}
for any local diffeomorphism $\mathfrak{f}\colon U\to V$.
\end{lem}
\begin{proof} First, we show that $U_\lambda\circ \mathfrak{f}^*=\mathfrak{f}^*\circ U_\lambda$, where by definition $\mathfrak{f}^*\omega=\omega\circ \mathfrak{f}$ for any differential form $\omega$. Indeed, we compute
\begin{align*}
(U_\lambda\circ \mathfrak{f}^*)\omega&{}=U_\lambda(\omega\circ \mathfrak{f})
=\sum_{i=1}^n \lambda^{-i}(\omega\circ \mathfrak{f})_{[i]}=\sum_{i=1}^n \lambda^{-i}\omega_{[i]}\circ \mathfrak{f}
=(\mathfrak{f}^*\circ U_\lambda)\omega.
\end{align*}
Hence, we get for $P\in {\rm Diff}(M,\Lambda T^*M)$
\[
\bigl(U_\lambda^\sharp\circ \mathfrak{f}^\sharp\bigr)(P)=U_\lambda\circ \mathfrak{f}^* \circ P\circ \mathfrak{f}_*\circ U_\lambda^{-1}=\mathfrak{f}^* \circ U_\lambda\circ P\circ U_\lambda^{-1}\circ \mathfrak{f}_*=\bigl(\mathfrak{f}^\sharp \circ U_\lambda^\sharp\bigr)(P).
\tag*{\qed}
\]
\renewcommand{\qed}{}
\end{proof}

As a direct consequence of equation~\eqref{eq:ulambdasharpdef}, we get the following

\begin{prop} \label{prop:pullbacklocalform} For any differential operator $P\in {\rm Diff}(M, \Lambda T^*M)$ of Agmon angle $\theta$, and for any $\lambda>0$:
\begin{gather}\label{eq:tildeu}
 \omega^{\rm Res}_{\log_\theta (U_\lambda^\sharp P)}=\omega^{\rm Res}_{\log_\theta P},\qquad \tilde \omega^{\rm Res}_{\log_\theta(U_\lambda^\sharp P)} =\lambda^{-n} \tilde \omega^{\rm Res}_{\log_\theta P},\qquad \tilde \omega^{ {\rm Res}}_{\log_\theta{({\mathfrak f}}_{ \lambda}^\sharp P)}=\lambda^{n} \tilde \omega^{\rm Res}_{\log_\theta P}\circ \mathfrak{f}_\lambda.
 \end{gather}
In particular, we get
	\begin{equation}\label{eq:tildeomegaflambda}
		\tilde \omega^{ {\rm Res}}_{\log_\theta({ U_\lambda^\sharp {\mathfrak f}}_{ \lambda}^\sharp P)}=\tilde \omega^{\rm Res}_{\log_\theta P}\circ {\mathfrak f}_{ \lambda}.\end{equation}
\end{prop}
\begin{proof} For any $\lambda>0$, we write
\begin{align*}
 (2\pi)^n\omega_{\log_\theta(U_{ \lambda}^\sharp P)}^{{ {\rm Res}}}(x) &{}= \biggl( \int_{\vert \xi\vert=1} {\rm tr}^{\Lambda T^*M}\bigl(\sigma_{ {-n}}\bigl(\log_\theta\bigl(U_{ \lambda}^\sharp P\bigr)\bigr)(x,\xi)\bigr) {\rm d}_S\xi\biggr) {\rm d}{ x^1}\wedge \dots \wedge {\rm d}{x^n}\\
 &{}\bui{=}{\eqref{eq:ulambdasharpdef}}\biggl( \int_{\vert \xi\vert=1} {\rm tr}^{\Lambda T^*M}\bigl(U_{ \lambda}^\sharp(\sigma_{ {-n}}( \log_\theta P)(x,\xi))\bigr) {\rm d}_S\xi\biggr) {\rm d}{ x^1}\wedge \dots \wedge {\rm d}{x^n}\\
 &{}=\biggl( \int_{\vert \xi\vert=1} {\rm tr}^{\Lambda T^*M}\bigl(U_{ \lambda} \sigma_{ {-n}}( \log_\theta P)(x,\xi)U_\lambda^{-1}\bigr) {\rm d}_S\xi\biggr) {\rm d}{ x^1}\wedge \dots \wedge {\rm d}{x^n}\\
 &{}= \biggl(\int_{\vert \xi\vert=1} {\rm tr}^{\Lambda T^*M}(\sigma_{ {-n}}( \log_\theta P)(x,\xi)) {\rm d}_S\xi\biggr) {\rm d}{ x^1}\wedge \dots \wedge {\rm d}{x^n}\\
 &{}=(2\pi)^n\omega_{ \log_\theta P}^{{ {\rm Res}}}(x).
\end{align*}
To prove the two other equalities, we also compute
\begin{align*}
		(2\pi)^n \tilde \omega^{\rm Res}_{\log_\theta(U_{ \lambda}^\sharp P)}(x) &{}= \int_{\vert \xi\vert=1} \big[ \bigl(\sigma_{ {-n}}\bigl(\log_\theta\bigl(U_{ \lambda}^\sharp P\bigr)\bigr)(x,\xi)\bigr) {\bf 1}_x\big]_{[n]} {\rm d}_S\xi\nonumber\\
 &{}\bui{=}{\eqref{eq:ulambdasharpdef}}\int_{\vert \xi\vert=1} \big[ \bigl(U_\lambda^\sharp (\sigma_{ {-n}}(\log_\theta( P))(x,\xi))\bigr) {\bf 1}_x\big]_{[n]} {\rm d}_S\xi\nonumber\\
 &{}=\int_{\vert \xi\vert=1} \big[U_\lambda (\sigma_{ {-n}}(\log_\theta( P))(x,\xi){\bf 1}_x)\big]_{[n]} {\rm d}_S\xi\nonumber\\
 &{}=\int_{\vert \xi\vert=1}\sum_{i=0}^n \big[U_\lambda [ \sigma_{-n}(\log_\theta( P))(x,\cdot) {\bf 1}_x ]_{[i]}\big]_{[n]} {\rm d}_S\xi\nonumber\\
&{}=\int_{\vert \xi\vert=1} \sum_{i=0}^n \big[\lambda^{-i}[ \sigma_{-n}(\log_\theta( P))(x,\cdot) {\bf 1}_x ]_{[i]}\big]_{[n]} {\rm d}_S\xi\nonumber\\
&{}=\lambda^{-n}\int_{\vert \xi\vert=1} [\sigma_{-n}(\log_\theta( P))(x,\cdot){\bf 1}_x]_{[n]} {\rm d}_S\xi\nonumber\\
 &{}= \lambda^{-n} (2\pi)^n \tilde \omega^{\rm Res}_{ \log_\theta( P)}(x).
\end{align*}
To prove the last equality in \eqref{eq:tildeu}, we use equality~(\ref{eq:sigmaflambda}) to write
\begin{align*}
(2\pi)^n \tilde \omega^{\rm Res}_{\log_\theta (\mathfrak{f}_\lambda^\sharp P)}(x)&{}=\int_{\vert \xi\vert=1} \big[ \sigma_{-n}\bigl( { {\mathfrak f}}_{ \lambda}^\sharp \log_\theta P\bigr)(x,\xi) {\bf 1}_x \big]_{[n]} {\rm d}_S\xi\\
&{}=\int_{\vert \xi\vert=1} \big[ \bigl(\sigma_{ {-n}}({ \log_\theta P}) \bigl({ { {\mathfrak f}}_{ \lambda}(x)},({\mathfrak f}_{ \lambda}^*)^t (\xi) \bigr) \bigr) {\bf 1}_x \big]_{[n]} {\rm d}_S\xi \\
&{}=\int_{\vert \xi\vert=1} \big[ \bigl(\sigma_{ {-n}}({ \log_\theta P}) \bigl({ { {\mathfrak f}}_{ \lambda}(x)},\lambda^{-1}\xi \bigr) \bigr) {\bf 1}_x \big]_{[n]} {\rm d}_S\xi \\
&{}=\lambda^n (2\pi)^n \tilde \omega^{\rm Res}_{\log_\theta P}({\mathfrak f}_{ \lambda}(x)).
\end{align*}
Finally, equality~\eqref{eq:tildeomegaflambda} is obtained by combining the last two identities in \eqref{eq:tildeu}. This gives the statement.
\end{proof}

\subsection{The Wodzicki residue density versus a local Berezin type density } 
In this paragraph, we enhance the well-known algebraic identity \eqref{eq:supertraceformula} to a lesser known identity of local densities on spin manifolds. Let now $(M,g)$ be a spin manifold of even dimension~$n$ and let $\Sigma M$ be its spinor bundle. The morphism $c^g$ defined in \eqref{eq:cliffend} induces on a differential operator $P=\sum_{|\gamma|\leq m} P_\gamma D^\gamma$ in ${\rm Diff}(M,\Sigma M)$ of order $m$, the operator $c^g(P)$ given by
\begin{equation}\label{eq:cgp}
 c^g(P):=\sum_{|\gamma|\leq m} c^g ( P_\gamma) D^\gamma
\in {\rm Diff}(M, \Lambda T^*M),
\end{equation}
where we have used the identification $\Cl(TM)\otimes\mathbb{C}\simeq \operatorname{End}(\Sigma M)$ as in Proposition~\ref{pro:spinorbundle}. Clearly, the operator $c^g(P)$ has the same order as $P$. In order to find the relation between~$\tilde\omega^{\rm Res}$ and~$\omega^{\rm Res}$, we need the following lemma:

\begin{lem} \label{lem:sigmalo} For any differential operator $P\in {\rm Diff}(M,\Sigma M)$ of Agmon angle $\theta$ and order $m$, the operator $c^g(P)$ has also an Agmon angle $\theta$. Also, we have that
\[\sigma_{-j}(\log _\theta (c^g(P)))(x,\xi)=c^g(\sigma_{-j}(\log _\theta(P))(x,\xi)),\qquad \forall j\in \Z_{\geq 0}. \]
\end{lem}

\begin{proof} From the injectivity of $c^g$ we easily deduce that the set of eigenvalues of $c^g(\sigma_L(P)(x,\xi))$ (resp.\ $c^g(P)$) is a subset of the one of $\sigma_L(P)(x,\xi)$ (resp.\ $P$). Thus, an Agmon angle $\theta$ for $P$ is also one for $c^g(P)$.
The second part of the assertion can be proved along the same lines as the proof of equation~\eqref{eq:ulambdasharpdef} with $c^g$ playing the role of $U_\lambda^\sharp$.
\end{proof}

By choosing $a=\sigma_{-n} (\log_\theta (P)) (x, \xi)$ in~\eqref{eq:supertraceformula}, for any differential operator $P\in {\rm Diff}(M,\Sigma M)$ which is $\Z_2$-grading, we get that

\begin{prop}\label{prop:tildeomegavsomega}
For any differential operator $P\in {\rm Diff}(M,\Sigma M)$ with Agmon angle $\theta$, which is even for the $\Z_2$-grading $\Sigma M=\Sigma^+ M\oplus \Sigma^- M$, we have
\[\widetilde \omega^{\rm Res}_{\log_\theta (c^g(P))}(x)=j_g(x)(-2{\rm i})^{-n/2}\omega^{\rm sRes}_{\log_\theta (P)}(x),\]
where $j_g(x)=\sqrt{\det  (g_{ij}(x))}$.
\end{prop}
\begin{proof} Using Lemma~\ref{lem:sigmalo} for $j=n$, we compute
\begin{align*}
\widetilde \omega^{\rm Res}_{\log_\theta (c^g(P))}(x)&{}=\frac{1}{(2\pi)^n}\int_{|\xi|=1} [ \sigma_{-n} (\log_\theta (c^g(P))) (x, \xi) {\bf 1}_x ]_{[n]} {\rm d}_S\xi\\
&{}=\frac{1}{(2\pi)^n}\int_{|\xi|=1} [ c^g(\sigma_{-n} (\log_\theta (P)) (x, \xi)) {\bf 1}_x ]_{[n]} {\rm d}_S\xi\\
&{}=\frac{1}{(2\pi)^n}\int_{|\xi|=1} [ {\bf s}^g(\sigma_{-n} (\log_\theta (P)) (x, \xi)) ]_{[n]} {\rm d}_S\xi\\
&{}=\frac{1}{(2\pi)^n}\int_{|\xi|=1} (T\circ {\bf s}^g)(\sigma_{-n} (\log_\theta (P)) (x, \xi) ){\rm e}^1\wedge\dots \wedge {\rm e}^n {\rm d}_S\xi\\
&{}\bui{=}{\eqref{eq:supertraceformula}}\frac{j_g(x)(-2{\rm i})^{-n/2}}{(2\pi)^n}\biggl(\int_{|\xi|=1}{\rm str} (\sigma_{-n}(\log_\theta (P)) (x, \xi) ) {\rm d}_S\xi \biggr) {\rm d}x^1\wedge\dots\wedge {\rm d}x^n\\
&{}=j_g(x)(-2{\rm i})^{-n/2} \omega^{\rm sRes}_{\log_\theta (P)}(x).
\end{align*}
Here, we use the fact that $P$ is $\Z_2$-graded, meaning that $\sigma_{-j}(\log _\theta (P))(x,\xi)$ is in $\operatorname{End}(\Sigma_x^{\pm} M)\simeq \Cl(T_xM)^+\otimes \C$ and, thus, equation~\eqref{eq:supertraceformula} is applied.
\end{proof}

\begin{rk}
As a consequence of Proposition~\ref{prop:tildeomegavsomega}, for a differential operator $P\in \Psi_{\rm cl}(M, \Sigma M)$ with Agmon angle $\theta$, $\widetilde \omega^{\rm Res}_{\log_\theta (c^g(P))}$ does define a global density since $\omega^{\rm sRes}_{\log_\theta (P)}$ does.
\end{rk}

\begin{cor}\label{cor:relationomegatilde}
For any differential operator $P\in {\rm Diff}(M,\Sigma M)$ of Agmon angle $\theta$ and order~$m$ which is even for the $\Z_2$-grading $\Sigma M=\Sigma^+ M\oplus \Sigma^- M$, we have
\[\tilde \omega^{{\rm Res}}_{ \log _\theta (U_\lambda^\sharp {\mathfrak f}_{ \lambda}^\sharp ( c^g(P)))}=(j_g\circ {\mathfrak f}_{\lambda}) (-2{\rm i})^{-n/2}\omega^{ {\rm sRes}}_{\log _\theta (P)}\circ \mathfrak{f}_\lambda.\]
\end{cor}

\begin{proof} By (\ref{eq:tildeomegaflambda}) applied to the differential operator $c^g(P)$, we write
\begin{equation*}
		\tilde \omega^{\rm Res}_{\log _\theta (c^g(P))}\circ {\mathfrak f}_{ \lambda} = \tilde \omega^{ {\rm Res}}_{{ \log _\theta (U_\lambda^\sharp {\mathfrak f}}_{ \lambda}^\sharp ( c^g(P)))}.
 \end{equation*}
The statement then follows from Proposition~\ref{prop:tildeomegavsomega} at the point $\mathfrak{f}_\lambda(\cdot)$.
\end{proof}

\section{The geometric set-up} 
In this section, we review the geometric set up underlying Getzler rescaling. Specifically, in the language of \cite[Section~1.1]{DS1},
we deform the manifold $M$ to a manifold $\M$ via a deformation to the normal cone to a given point $p$, and pull back the Riemannian metric $g$ on the manifold under the canonical projection $\hat \pi\colon \M\to M$ to a family $\{g_\lambda\}_{\lambda>0}$ of dilated metrics (see \eqref{eq:glambda}). This family will play a crucial role when deforming operators.
\subsection{Deformation to the normal cone to a point}
For an embedding $M_0\hookrightarrow M$ of two manifolds, the {\it deformation to the normal cone} is defined as
\begin{equation*}
D(M_0, M):= (M\times \R_+ )\cup ({\mathcal N } M_0\times\{0\}),
\end{equation*}
where ${\mathcal N } M_0$ is the total space of the normal bundle to $M_0$ in $M$. The deformation to the normal cone extended to the embedding of the base of a groupoid into the groupoid gives rise to the tangent groupoid introduced by Connes \cite{C} which proves useful in the context of manifolds with singularities. Here, choosing a reference point $p\in M$ fixed throughout the paper, we take $M_0=\{p\}$ so that ${\mathcal N } M_0=T_{p}M$, in which case the deformation amounts to replacing $M$ by the {\it deformed manifold around $p$} defined as
\[
\mathbb M:=(M\times \R_+ )\cup (T_{p}M\times\{0\}).
\]
 The gluing of the two parts, namely $M\times \R_+$ and $T_{p}M\times\{0\}$ is carried out via the local diffeomorphism $\mathfrak{f}_\lambda$ described in \eqref{eq:flambdaUplambda} as follows. We build the map (denoted by $\Theta$ in \cite[Section~3.1]{DS2}, but here we adopt the notations of \cite{vEY})
 \[
 \Exp_{p}\colon \ T_pM\times\R\supset \B_{p, r }\longrightarrow \M
 \]
defined on
 \[
 \B_{p,r}:= (\{({\mathbf x}, \lambda)\in T_{p}M\times \R_+, {\mathbf x}\in B_{r/
\lambda}\} )\cup (T_{p}M\times \{0\} )
\]
 by the identity map on $T_{p}M\times \{0\}$ and on the remaining part of $\B_{p,r}$ as follows:
 \begin{align}
 \Exp_{p}\colon \ \{ ({\mathbf x}, \lambda)\in T_{p}M\times \R_+, {\mathbf x}\in B_{r/
\lambda}\}&\longrightarrow U_{p}\times \R_+ \subset\mathbb M,\nonumber\\
({\mathbf x}, \lambda)&\longmapsto (\exp_{p}(\lambda {\mathbf x}), \lambda), \qquad {\rm for}\ \lambda>0.\label{eq:Fx0}
\end{align}

We consider the open set in $\M$ \cite[Section~1.1]{DS1} \[W_{p}:= ( U_{p}\times \R_+)\cup (T_{p} M\times\{0\} )\subset \M .\]

The deformed manifold $\M$ is endowed with the smooth structure for which $\Exp_{p}$ is a diffeomorphism, and which restricts to the standard smooth structure on {$M\times \R_+$} (we refer the reader to \cite[above Lemma~4.3]{H} for further details). Via $\Exp_p$ the point $(x=\exp_p({\bf x}), \lambda)$ is identified with the point $\mathfrak{f}_\lambda(x)={\rm exp}_p(\lambda {\bf x})$ and the point $p$ is identified with ${\bf x}$. We refer to the coordinates given by (\ref{eq:Fx0}) as the {\it $\lambda$-rescaled exponential coordinates}. To recover the manifold $M$ from the deformed manifold $\M$, we consider the projection map
\begin{align*}
\hat \pi\colon \ \M & \bui{\longrightarrow}{p_1} M\times\R_{\geq 0} \bui{\longrightarrow}{\pi} M,\\
(x, \lambda)& \longmapsto (x, \lambda) \longmapsto x \quad {\rm if} \ \lambda>0, \\
({\mathbf x},0) & \longmapsto (p,0) \longmapsto p.
\end{align*}

With $\M$ endowed with the smooth structure described above, the map $\hat\pi$ is smooth allowing to pull-back the geometry on $M$ to $\M$. For any section $s$ of a vector bundle $E$ over $M$, its pull-back is a section of the pull-back bundle $\E:=\hat \pi^* E\subset \M\times E$ over $\M$ given by
\[
(\hat\pi^* s)(x,\lambda)=s(x),\qquad \forall \lambda>0 ,\quad \forall x\in M
\]
and
\[
(\hat\pi^* s)({\mathbf x},0)=s(p),\qquad \forall {\mathbf x}\in T_pM.
\]

In particular, the tangent bundle $TM\to M$ is pulled back to
\begin{gather*}
\hat \pi^* TM=\{(x,\lambda,y,u)\in M\times\R_+\times TM\mid x=y\}\\
\hphantom{\hat \pi^* TM=}{}
\cup \{({\bf x},0,y,u)\in T_pM\times \{0\}\times TM\mid p=y\}.
\end{gather*}

Also, the local diffeomorphism (\ref{eq:Fx0}) induces the isomorphism of vector bundles (see \cite[Remark~3.4\,(e)]{DS2} and \cite[pp.~67--68]{F})
\begin{align*}
	(TM\times\R_+)\cup (T_pM\times \{0\})&\longrightarrow T\M,\\
	 (x, u, \lambda)&\longmapsto (x, \lambda u= h_\lambda(u), \lambda) \quad {\rm if}\ \lambda>0,\\
	 (\mathbf{x},0)&\longmapsto \mathbf{x}\in T_{p}M \quad {\rm if}\ \lambda=0.
\end{align*}
 Now, in the local exponential chart (\ref{eq:Fx0}) of $\M$, the pull-back of a section $s$ on $E$ is the map $\hat\pi^*s\circ\Exp_p\colon \B_{p,r}\longrightarrow \E$ that can be read as
\[(\hat\pi^*s\circ\Exp_p)({\bf x},\lambda)=(s\circ \hat\pi)({\rm exp}_p(\lambda {\bf x}),\lambda)=s({\rm exp}_p(\lambda {\bf x}))\]
for any $\lambda>0$ and ${\bf x}\in B_{r/
\lambda}$. Also, on the remaining part of $\B_{p,r}$, we have
\[(\hat\pi^*s\circ\Exp_p)({\bf x},0)=(s\circ\hat\pi)({\bf x},0)=s(p).\]
Therefore, by taking ${\bf x}\in B_r\subset B_{r/
\lambda}$ for $\lambda>0$ small enough and identifying it with the point $x:={\rm exp}_p{\bf x}\in U_p$, we write that
\begin{equation*}
(\hat \pi^* s)(x,\lambda)\overset{\text{\bf Exp}_p}{=}(s\circ \mathfrak{f}_\lambda)(x),
 \end{equation*}
for $x\in U_p$.

\subsection[Tensor bundles pulled back by hat pi]{Tensor bundles pulled back by $\boldsymbol{\hat\pi}$} 
 Coming back to the deformation to the normal cone, the tensor bundle $T_r^qM:=TM^{\otimes q}\otimes T^*M^{\otimes r}$ is pulled back to $\hat \pi^* T^q_rM \longrightarrow \M$ and a tensor field $t$ written in a normal geodesic coordinates chart with coordinates $X$ at a point $x\in U_p$ as
\begin{equation}\label{eq:tensor}
t(x)=\sum t_{j_1\dots j_r}^{i_1 \dots i_q}\frac{\partial}{\partial x^{i_1}} \otimes\dots\otimes \frac{\partial}{\partial x^{i_q}} \otimes {\rm d}x^{j_1}\otimes\dots\otimes {\rm d}x^{j_r}\bigg|_x
\end{equation}
is pulled back to
\begin{align*}
&(\hat \pi^* t)(x,\lambda) \overset{\text{\bf Exp}_p}{=} t\circ \mathfrak{f}_\lambda(x)\nonumber\\
&\qquad{}=\sum\! \bigl(t_{j_1\dots j_r}^{i_1 \dots i_q}\circ \mathfrak{f}_\lambda\bigr)\! \biggl(\frac{\partial}{\partial x^{i_1}}\circ \mathfrak{f}_\lambda\biggr) \otimes\dots\otimes \biggl(\frac{\partial}{\partial x^{i_q}}\circ \mathfrak{f}_\lambda\biggr) \otimes \bigl({\rm d}x^{j_1}\circ \mathfrak{f}_\lambda\bigr) \otimes\dots\otimes \bigl({\rm d}x^{j_r}\circ \mathfrak{f}_\lambda\bigr)\biggl|_x,
\end{align*}
for small enough $\lambda\geq 0$. Combining \eqref{eq:derivativedifferent} with (\ref{eq:getzlerrescaling}), we deduce that
\begin{equation*}
(\hat \pi^* t)(x,\lambda)\overset{\text{\bf Exp}_p}{=}\lambda^{q-r}(\mathfrak{f}_\lambda^*t)(x),
\end{equation*}
for any small enough $\lambda>0$ and $(\hat \pi^* t)(x,0)\overset{\text{\bf Exp}_p}{=}t(p)$, for $\lambda=0$. Specialising to $q=0$ and $r=2$, yields that the local description of the pull-back of the metric $g$ on $M$, viewed as a covariant two tensor is
\begin{equation}\label{eq:pistarg}
(\hat \pi^* g)(x,\lambda)\overset{\text{\bf Exp}_p}{=}\lambda^{-2} ({\mathfrak f}_\lambda^*g)(x)
\end{equation}
for small enough $\lambda>0$ and it is $g(p)$ for $\lambda=0$. It is therefore natural to introduce \begin{equation}\label{eq:glambda}g_\lambda:=\lambda^{-2} \mathfrak{f}_\lambda^* g,
\end{equation}
so that at any point in $U_p$, we have
$\left(g_\lambda\right)_{ij}(x)=g_{ij}(\mathfrak{f}_\lambda(x)).$ As a consequence of the last identity and with the help of the Koszul formula, the Christoffel symbols $\Gamma_{i j}^k (\cdot, g):=g\bigl(\nabla_{\frac{\partial}{\partial x^i}} \frac{\partial}{\partial x^j}, \frac{\partial}{\partial x^k}\bigr)$ satisfy for any small enough positive $\lambda$
\begin{equation*}
 \Gamma_{i j}^k (\cdot, g_\lambda) =\lambda \Gamma_{i j}^k (\mathfrak{f}_\lambda(\cdot),g).
 \end{equation*}
Similarly, the Christoffel symbols $\tilde\Gamma_{ls}^t(\cdot,g) :=g(\nabla_{e_l}e_s,e_t)$ read in an orthonormal frame obtained by parallel transport along the geodesic curves, satisfy
\begin{equation}\label{eq:symbchriglamdabis}
\tilde\Gamma_{ls}^t(\cdot,g_\lambda) =\lambda \tilde\Gamma_{ls}^t( \mathfrak{f}_\lambda(\cdot), g).
\end{equation}

\section{Geometric differential operators}
In this section, we define the notion of geometric polynomials with respect to a given metric, as smooth sections (Definition~\ref{defn:geometricpolE}) of a given vector bundle in terms of the corresponding vielbeins (see Appendix~\ref{sec:appendix1}). To these polynomials, we assign an order called Gilkey order, inspired by Gilkey's ``order of jets" in the context of his invariance theory \cite[Section~2.4]{G}, see also~\mbox{\cite[Section~3]{MP}}, both of which use jets of metrics. Whereas geometric polynomials are defined in terms of the jets of the vielbeins and hence of the metric tensor, the Gilkey order does not depend on the choice of metric. We call a differential operator geometric if its coefficients written in a local trivialisation are geometric polynomials (Definition~\ref{defn:geometricdiffop}). In Proposition~\ref{pro:flambdas}, we show that a geometric polynomial with respect to $g$ transforms under a contraction $\mathfrak f_\lambda$ to a~geometric polynomial with respect to $g_\lambda$. In Proposition~\ref{pro:geometricopes}, we show a similar property for a~geometric differential operator.

\subsection{Valuation of local sections}
Let us recall some basic facts on the jets of a vector bundle. Given any vector bundle $(E,\pi, M)$ where $\pi\colon E\to M$ is the orthogonal projection, we let $\Gamma(E):=C^\infty(M,E)$ be the vector space of sections of $E$ and $\Gamma_p(E)$ be the stalk\footnote{Let $S_U$ denote the set of local sections of $E$ defined on an open subset $U$ of $M$ containing $p$. The stalk of local sections at the point $p$ is the set of $\sim$ equivalence classes where for two elements $s_U$ in $S_U$ and $s_V$ in $S_V$, $s_U\sim s_V$ if and only $s_U$ and $s_V$ coincide in some neighborhood of $p$.} of local sections at a point $p$. Two local sections~$s$ and~$s'$ in~$\Gamma_p(E)$ have the same $r$-jet ($r\in \mathbb{Z}_+$) at $p$ if
\[(D^{\gamma}s) \vert _p=(D^{\gamma}s') \vert _p\]
for any multiindex $\gamma$ such that $0\leq |\gamma|\leq r$. The relation
\[s\sim s'\iff s\ \text{and}\ s'\ \text{have the same $r$-jet at $p$}\]
defines an equivalence relation and we denote by $j_p^rs$ the equivalence class of $s$. The integer $r$ is called the order of the jet. The set
\[J^r(E):=\big\{j_p^rs\mid p\in M,\, s\in \Gamma_p(E)\big\}\]
is a manifold, called the $r$-th jet manifold of $\pi$. The triple $(J^r(E),\pi_r,M)$ is a fiber bundle where $\pi_r\colon J^r(E)\to M$; $j_p^rs\mapsto p$ and, in local coordinates,
\[j_p^rs=(s(p),D^{\gamma}s|_p; 1\leq |\gamma|\leq r),\]
which can be locally represented by the polynomial $\sum_{|\gamma|\leq r} D^{|\gamma|} s|_p X^\gamma$ (here in the variable $X$).
The reference vector bundle $E$ will often be implicit only when needed shall we mention it.
\begin{defn} Given a normal geodesic coordinate system $X=\bigl(x^1,\dots,x^n\bigr)$ at a point $p$ and for a non negative integer $r$, the $r$-valuation of a local section $s\in\Gamma_p(E)$ is defined by
\[{\rm val}^r_{X,p}(s)=\min \big\{ \vert\gamma\vert\leq r, D_X^{\gamma}s \big \vert _p\neq 0\big\}\]
with the notation of \eqref{eq:dgammax} provided such a minimum exists. Otherwise, following the usual convention we set ${\rm val}^r_{X,p}(s)=+\infty$.
Correspondingly, we define the valuation of $s$ as being
\begin{equation}\label{eq:valp}
{\rm val}_{X,p}(s)=\min _{r\in \Z_{\geq 0}} {\rm val}^r_{X, p}(s)=\min \big\{\vert\gamma\vert\mid D^{\gamma}s \big \vert _p\neq 0\big\} \in [0,+\infty].
\end{equation}
\end{defn}

\begin{ex} We choose $E=T^*M\otimes T^*M$, and view $g$ as a section of $E$ trivialised above~$U_p$ by means of normal geodesic coordinates on $U_p$. In that trivialisation, the expansion of the metric around a point $p$ is given by
\begin{equation}\label{eq:metricnormalcoord}
g_{ij}(x)=\delta_{ij}-\frac{1}{3}R_{ikl j}\bigg\vert_p x^k x^l-\frac{1}{6} R_{ikl j;m}\bigg\vert_p x^k x^l x^m +O\bigl(|x|^4\bigr),
\end{equation}
where $R_{iklj;m}=(\nabla R)_{m iklj}.$ Therefore, in this trivialisation, the valuation of $g-{\rm Id}$ is at least~$2$. In contrast, in the trivialisation of $E$ obtained by parallel transport, the valuation of $g-{\rm Id}$ is~$+\infty$.

Similarly the expansion of the inverse is given by
 \begin{equation}\label{eq:metricnormalcoordbis}
 g^{ij}(x) =\delta_{ij}+\frac{1}{3}R_{ikl j}\bigg\vert_p x^k x^l+\frac{1}{6} R_{ikl j;t}^g\bigg\vert_px^k x^l x^t +O\bigl(|x|^4\bigr),
 \end{equation}
 so that, in these coordinates and with a slight abuse of notation, the valuation of $g^{-1}-{\rm Id}$ is at least $2$. Combining equations~(\ref{eq:BBt}) with (\ref{eq:metricnormalcoord}) (resp.\ (\ref{eq:AAt}) with (\ref{eq:metricnormalcoordbis})) yields the following expansions \cite[equation~(11)]{AGHM}
\begin{equation*}
a_i^l(x, g)=\delta_{il}-\frac{1}{6}R_{ijk l}( p) x^j x^k-\frac{1}{12}\nabla_t R_{ijk l}x^j x^k x^t+O\bigl(|x|^4\bigr)\end{equation*}
 and
 \begin{equation*}
 b_l^i(x, g)=\delta_{il}+\frac{1}{6}R_{ljk i}( p) x^j x^k+\frac{1}{12}\nabla_t R_{ljk i}x^j x^k x^t+O\bigl(|x|^4\bigr).
 \end{equation*}
Hence, in the same way as before, the valuation of $A-{\rm Id}$ (resp.\ $B-{\rm Id}$, see Appendix~\ref{sec:appendix1}) is at least $2$ as well.

Finally, using the Koszul formula combining with (\ref{eq:metricnormalcoord}) and the properties of the curvature operator, the Christoffel symbols $\Gamma_{ij}^k (\cdot, g)=g\bigl(\nabla_{\frac{\partial}{\partial x^i}} \frac{\partial}{\partial x^j}, \frac{\partial}{\partial x^k}\bigr)$ have the following Taylor expansion at point $p$ in the normal geodesic coordinates
\[\Gamma_{ij}^k(x,g)=\frac{1}{3}(R_{i k l j}(p)+R_{i l k j}(p)) x^l+ O\bigl(\vert x\vert^2\bigr).\]
Also, the Christoffel symbols in an orthonormal frame $\tilde\Gamma_{ls}^t(\cdot,g)=g(\nabla_{e_l}e_s,e_t)$ have a similar Taylor expansion
\begin{equation}\label{eq:Gammatilde}
\tilde\Gamma_{ls}^t(x,g)=-\frac{1}{2}R_{li st}(p) x^i+O\bigl(\vert x\vert^2\bigr),
\end{equation}
which shows that both Christoffel symbols have valuation at least $1$.
\end{ex}

\subsection{Polynomial expressions in the jets of the vielbeins}

We consider a rank $k$ vector bundle $E\to M$ equipped with an affine connection. We trivialise the bundle $E$ over an exponential neighborhood $U_p$ of $p$ using geodesic normal coordinates $\bigl(x^1, \dots, x^n\bigr)$ at a point $x$ in $U_p$ by identifying the fibre $E_x$ above $x=\exp_p( {\bf x})\in U_p$ with the fibre $E_p$ at point $p$ via the parallel transport along geodesics $c(t)=\exp_p(t \bf x)$, ${\bf x}\in T_pM$. We fix a basis $(s_1(p), \dots, s_k(p))$ of $E_p$, which is then transported to $(s_1(x,g), \dots, s_k(x, g))$. In this trivialisation, sections of $E$ may be viewed as smooth functions on $U_p$ with valued in the fixed fibre $E_p$.
\begin{defn}\label{defn:geometricpolE}
We call a local section $s$ of $E$ over $U_p$ a {\it geometric monomial $($resp.\ polynomial$)$} with respect to some metric $g$, if when $s=\sum_{j=1}^k \alpha_j s_j$ is written in the local trivialisation $s_j(\cdot, g)$, $j=1, \dots, k$ of $E$ above $U_p$, the coordinates $\alpha_j(\cdot,g)$ are monomials (resp.\ polynomials) in the jets of vielbeins $A_p(\cdot,g)$ and $B_p(\cdot,g)$ for the metric $g$ (resp.\ linear combinations of monomials), namely if they are (resp.\ linear combinations of) expressions of the form
\begin{equation}\label{eq:polynomialsvielbeins}
\prod_{q=1}^{S_j} D^{\beta_{q}^j}\bigl((a_j)_{i_q}^{t_q}(\cdot,g)\bigr) D^{\gamma_q^j}\bigl((b_j)^{l_q}_{n_q}(\cdot,g)\bigr),
\end{equation}
such that $\sum_{q=1}^{S_j}\big\vert \beta_q^j\big\vert +\big\vert \gamma_q^j\big\vert$ is independent of $j$. In this case, we shall write $ s(\cdot,g)$ for an expression of the type (\ref{eq:polynomialsvielbeins}) and call ${\rm ord}^{\rm Gil}( s):=\sum_{q=1}^{S_j}\big\vert \beta_q^j\big\vert +\big\vert \gamma_q^j\big\vert$ its {\it Gilkey order}.
\end{defn}

One observes that the notion of geometric polynomial and its Gilkey order is invariant under transformations $g\mapsto f^*({\rm e}^\varphi g)$ of the metric $g$, where $f$ is a diffeomorphism on $M$ and $\varphi $ is a~smooth function on $M$.

\begin{rk} \label{lem:subbundlemetric}	When $E$ is a subbundle of the tensor bundle, we can alternatively trivialise it over the exponential neighborhood $U_p$ of $p$ using geodesic normal coordinates $\bigl(x^1, \dots, x^n\bigr)$ at a point in $U_p$.
By (\ref{eq:viedbein}), we have $\frac{ \partial}{\partial x^i}=\sum_{\ell=1}^n a_i^\ell (\cdot, g) e_\ell(\cdot,g)$ and ${\rm d}x^i=\sum_{m=1}^n b_l^i(\cdot, g) {\rm e}^l(\cdot,g) $, where $( e_1(\cdot, g), \dots, e_n(\cdot, g))$ is the basis of $TM$ obtained by parallel transport of some (fixed) orthonormal basis of $ T_pM$. Inserting these relations in
(\ref{eq:tensor}) yields an expression of
\[
\mathbf t(\cdot)=\sum t_{j_1\dots j_r}^{i_1 \dots i_q} P_{i_1\dots i_q }^{\ell_1\dots \ell_q} Q^{j_1\dots j_q }_{m_1\dots m_r} e_{\ell_1}(\cdot, g)\otimes\dots\otimes e_{\ell_q}(\cdot, g)\otimes {\rm e}^{m_1}(\cdot, g)\otimes \dots \otimes {\rm e}^{m_r}(\cdot, g),
\]
where $P_{i_1\dots i_q }^{\ell_1\dots \ell_q}$ and $Q^{j_1\dots j_q }_{m_1\dots m_r} $ are linear combinations of expressions of the form (\ref{eq:polynomialsvielbeins}). Thus, the coordinates $ t_{j_1\dots j_r}^{i_1 \dots i_q} $ of ${\mathbf t}$ in
(\ref{eq:tensor}) are linear combinations of expressions of the form (\ref{eq:polynomialsvielbeins}) if and only if its coordinates $ \tilde t_{m_1\dots m_r}^{\ell_1 \dots \ell_q} $ in the basis $ e_{\ell_1}(\cdot, g)\otimes\dots\otimes e_{\ell_q}(\cdot, g)\otimes {\rm e}^{m_1}(\cdot, g)\otimes \dots \otimes {\rm e}^{m_r}(\cdot, g)$ are also linear combinations of expressions of the form (\ref{eq:polynomialsvielbeins}).
Consequently, we can use either trivialisation in this case.
\end{rk}

Here are first examples of geometric polynomials.
\begin{ex}\label{ex:geompolTM}\quad
\begin{enumerate}\itemsep=0pt
\item Take $E=T^*M\otimes_s T^*M$. The metric $g$, which is a local section of $E$, is a geometric monomial with respect to the metric $g$ of Gilkey order zero since its coordinates in the basis ${\rm d}x^i \otimes {\rm d}x^j$ induced by the normal geodesic coordinates $x^1, \dots, x^n$ read
$g_{ij}(\cdot)=	\sum_{l=1}^n a_i^l(\cdot, g)a_j^l(\cdot, g)$ (see equation~\eqref{eq:AAt}). So is its inverse $g^{-1}$ a geometric monomial of Gilkey order zero since
$ g^{ij}(\cdot)=\sum_{l=1}^n b_l^i(\cdot,g) b_l^j(\cdot,g)$ (see equation~\eqref{eq:BBt}).
\item Take $E= T^*M\otimes T^*M \otimes TM$. The Christoffel symbol which is a local section of $E$, is a~geometric polynomial with respect to the metric $g$ since its coordinates $\Gamma_{ij}^k(\cdot,g)$ in the basis ${\rm d}x^i \otimes {\rm d}x^j\otimes \frac{{\rm d}}{{\rm d}x^k}$ induced by the normal geodesic coordinates $x^1, \dots, x^n$ read by Koszul's formula
\begin{equation}\label{eq:Koszul}\Gamma_{i j}^k(\cdot,g)=\sum_{l=1}^n g^{k l}(\cdot)\left(\partial_{x^i}(g_{jl}(\cdot))+\partial_{x^j}(g_{il}(\cdot))-\partial_{x^l}(g_{i j}(\cdot))\right),
\end{equation}
is a polynomial in the jets of vielbeins of Gilkey order one.
	
\item Similarly, the Christoffel symbols $\tilde\Gamma_{ls}^t(\cdot,g)=g(\nabla_{e_l}e_s,e_t)$ written in the orthonormal frame $(e_1(\cdot, g), \dots, e_k(\cdot, g))$ obtained by parallel transport as in Remark~\ref{lem:subbundlemetric} read as (use Einstein convention)
\begin{equation}\label{eq:relationschristofellsymbols}
\tilde\Gamma_{l s}^t (\cdot, g)=b_l^i(\cdot, g) b_t^j(\cdot, g) b_s^k(\cdot, g) \Gamma_{i k}^j (\cdot, g)+ b_l^i(\cdot, g) b_t^j(\cdot, g) \partial_{x_i}\bigl(b_t^k(\cdot, g)\bigr) g_{k j}(\cdot),
\end{equation}
are polynomials in the jets of the vielbeins of Gilkey order one.
\end{enumerate}
\end{ex}

\begin{rk}
Since jets are compatible with composition and differentiation, geometric monomials form an algebra stable under differentiation.
\end{rk}
\begin{rk}
	Due to equations~\eqref{eq:AAt} and \eqref{eq:BBt} which relate the metric to the vielbeins, the class of polynomials we single out in Example~\ref{ex:geompolTM},
	is consistent with the classes of polynomials in the jets of the metric considered in \cite{ABP}, \cite[Theorem 1.2]{E} and \cite[equation~(2.4.3)]{G}. There, the polynomials depend on the metric tensor, its inverse -- or its inverse determinant, see \cite[formula, item~1, p.~282]{ABP} -- and the derivatives of the metric tensor.
\end{rk}

\subsection{Geometric operators} In this subsection, we define geometric differential operators on vector bundles based on the definition of geometric polynomials.

\begin{defn}\label{defn:geometricdiffop}
 Let $E$ be a vector bundle over $M$ of finite rank equipped with an affine connection. We call a differential operator $P= \sum_{\vert \gamma\vert \leq m} P_\gamma D^\gamma$ in $ {\rm Diff}(M, E)$ of order $m$ {\it geometric} with respect to a metric $g$ if its coefficients
 $P_\gamma (x) \in \operatorname{End}(E_x)$, written in the basis obtained by parallel transport of some fixed basis of $E_p$, are geometric polynomials in the jets at $x$ of vielbeins such that for all $\gamma$
\begin{equation}\label{eq:geomopGi}
		{\rm ord}^{\rm Gi}(P_\gamma(x))={\rm ord}(P)- |\gamma|.
\end{equation}
\end{defn}
\begin{rk} As for geometric sections, when $E$ is a subbundle of the tensor bundle, we can alternatively trivialise it over the exponential neighborhood $U_p$ of $p$ using geodesic normal coordinates $\bigl(x^1, \dots, x^n\bigr)$ at a point in $U_p$.
\end{rk}

 \begin{ex}\label{ex:nablaXdiffgeom}
 For $E=\Lambda T^*M$, resp.\ $E=\Sigma M$, for any $X\in TM$,
 the covariant differentiation~$\nabla_X$ defines a geometric operator with respect to $g$ of order $1$.
 \begin{enumerate}\itemsep=0pt
\item When $E=\Lambda T^*M$, we express the covariant derivative on a differential form $\alpha=\sum_{I}\alpha_{I} {\rm d}x^I$ of degree $k$ in normal geodesic coordinates $\bigl(x^1, \dots, x^n\bigr)$ around $p\in M$ as follows (here $I=\{i_1<\dots<i_k\}$)
 	\begin{gather*}
 	 	\nabla_{\frac{\partial}{\partial x^i}}\alpha
= \sum_{I}\biggl({\frac{\partial}{\partial x^i}}\alpha_{I}\biggr) {\rm d}x^I\\
\hphantom{\nabla_{\frac{\partial}{\partial x^i}}\alpha =}{}
 + \sum_{t,I}\alpha_{I}\Biggl(\sum_{s,l=1}^n g\bigl(\nabla_{\frac{\partial}{\partial x^i}} {\rm d}x^{i_s},{\rm d}x^l\bigr)g_{tl}(\cdot)\Biggr) {\rm d}x^{i_1}\wedge\dots\wedge \underbrace{{\rm d}x^t}_{s^{\rm th}\text{{\rm -slot}}}\wedge\dots\wedge {\rm d}x^{i_k}.
 	\end{gather*}
This shows that $\nabla_{\frac{\partial}{\partial x^i}}$ is a geometric differential operator with respect to the metric $g$ of order $1$ whose zero-th order part $\sum_{m,l=1}^n g\bigl(\nabla_{\frac{\partial}{\partial x^i}} {\rm d}x^{i_s},{\rm d}x^l\bigr)g_{tl}(\cdot)$ has coefficients given by linear combinations of monomials (\ref{eq:polynomialsvielbeins}) with Gilkey order $1$.
 \item When $E=\Sigma M$ is equipped with the spin connection induced by the Levi-Civita connection, the corresponding $\operatorname{End}(\Sigma_p M)$-valued functions
\[{\rm e}^i \cdot_g\colon \ \sigma \mapsto {\rm e}^i \cdot_g \sigma,\qquad i=1, \dots, n\]
on $U_p$ are constant along the geodesics and hence in the trivialisation induced by parallel transport \cite[Lemma 4.14]{BGV}. In the normal geodesic coordinates $x^1, \dots, x^n$ on $U_p$, the Clifford multiplication operators ${\rm d}x^i \cdot_g$ are geometric operators since by (\ref{eq:viedbein}) they read
 $ {\rm d}x^i 	\cdot_g =\sum_{l=1}^n b_l^i(\cdot, g) {\rm e}^l \cdot_g $.

\item When $E=\Sigma M$, the spinorial connection $\nabla_{\frac{\partial}{\partial x^i}}$ acting on smooth functions from $U_p$ to $E_p$ reads
\begin{equation}\label{eq:connspin}
\nabla_{\frac{\partial}{\partial x^i}} =\frac{\partial}{\partial x^i} +\frac{1}{4}\sum_{s,t}\underbrace{g\bigl(\nabla_{\frac{\partial}{\partial x^i}} e_s,e_t\bigr)}_{\widetilde\Gamma_{i s }^t(\cdot,g)} {\rm e}^s\cdot_g {\rm e}^t\cdot_g
\end{equation}
and therefore defines a geometric differential operator with respect to $g$ of order $1$. Indeed, the coefficients $\widetilde\Gamma_{i s}^t(\cdot,g)$ are smooth real functions from $U_p$ and, by (\ref{eq:relationschristofellsymbols}), are polynomials in the jets of the vielbeins of Gilkey order $1$.
\end{enumerate}
\end{ex}

\begin{prop}\label{prop:productgeom}
The product of geometric differential operators with respect to
the metric $g$ is a geometric differential operator.
\end{prop}
\begin{proof}
Indeed, let
\[P:= \sum_{\vert\gamma\vert \leq {\rm ord} (P)} {P_{\gamma}D^{\gamma}}\qquad\text{and}\qquad Q:= \sum_{\vert\delta\vert \leq {\rm ord} (Q)} {Q_{\delta}D^{\delta}}\]
be two geometric operators. Their composition reads			
\[PQ= \sum_{\substack{\delta\\ |\gamma_1|+|\gamma_2|=|\gamma| } } P_{\gamma} (D^{\gamma_1} Q_{\delta}) D^{\gamma_2+\delta}.\]
We easily check that
\[{\rm ord}^{\rm Gi} (D^{\gamma_1} P_{\delta} )= {\rm ord}^{\rm Gi} (P_{\delta} )+\vert {\gamma_1}\vert,\]
so
\begin{align*}
{\rm ord}^{\rm Gi} (P_\gamma (D^{\gamma_1} Q_{\delta}) )+|\gamma_2+\delta|&{}= {\rm ord}^{\rm Gi}(P_\gamma)+{\rm ord}^{\rm Gi} (Q_{\delta} )+\vert {\gamma_1}\vert+\vert {\gamma_2}\vert+\vert {\delta}\vert\\
&{}= {\rm ord}^{\rm Gi}(P_\gamma)+{\rm ord}^{\rm Gi} (Q_{\delta} )+\vert {\gamma}\vert+\vert {\delta}\vert\\
&{}= {\rm ord} (P)+{\rm ord} (Q) = {\rm ord} (PQ).
\end{align*}
This finishes the proof.
\end{proof}
\begin{rk} In \cite{HY}, the authors assign a Getzler order to linear partial differential operators~$D$ acting on the smooth sections of the spinor bundle $\Sigma M$ over $M$ which can be expressed as a~finite sum of operators of the form $f\cdot D_1\cdots D_p$ where $f$ is a smooth function and each~$D_j$ is either a covariant derivative $\nabla_X$, or a Clifford multiplication operator $X\cdot_g$, or the identity operator. In our terminology, such an operator is geometric as the product of geometric operators~$\nabla_X$ and $X\cdot_g$ (see Example~\ref{ex:nablaXdiffgeom}). We will later see in Example~\ref{ex:Hodgelaplacian} that the exterior differential ${\rm d}$ is a geometric operator but not of the form $f\cdot D_1\cdots D_p$. Also, the notion of geometric operator generalises to classical pseudodifferential operators in requiring a condition similar to~(\ref{eq:geomopGi}) on the homogeneous components of the symbol, see \cite{MP}.
\end{rk}

In the following, we consider again a vector bundle $E$ over $M$ of rank $k$ equipped with a~connection $\nabla$.

\begin{lem} \label{ref:eq:patrgglambda} Let $E$ be a vector bundle over $M$ of rank $k$ and let $s(p)\in E_p$, for some fixed $p\in M$. We denote by $s(\cdot,g)$ the section in $\Gamma(E)$ obtained by parallel transport of $s(p)$ along the exponential curve $c(t)={\rm exp}_p(t{\bf x})$ corresponding to the metric $g$ with ${\bf x}\in B_r\subset T_pM$. We also denote by $\overline{s}(\cdot,g_\lambda)$ the section obtained by parallel transport of $s(p)$ along the exponential curve~$\gamma(t)$ corresponding to the metric $g_\lambda$. Then, we have
\[\overline{s}(\cdot,g_\lambda)=s(\mathfrak{f}_\lambda(\cdot),g).\]
\end{lem}
\begin{proof}
First, we notice that if $c(t)={\rm exp}_p(t{\bf x})$ is the exponential curve corresponding to the metric $g$ with ${\bf x}\in B_r\subset T_pM$, then the curve $\gamma\colon I\to U_p$ given by
\[\gamma(t):=\mathfrak{f}_\lambda\circ c(t)={\rm exp}_p(t\lambda{\bf x})\] is the exponential curve associated with the metric $\mathfrak{f}_\lambda^*g$ as well for the metric $g_\lambda$.
The section~$s(\cdot,g)\circ\mathfrak{f}_\lambda$ is parallel along the curve $\gamma(t)$ as a direct consequence from the fact that $s(\cdot,g)$ is parallel on $E$ along the curve $c(t)$. Now the initial condition and the uniqueness of the parallel transport allow to deduce the result.
\end{proof}

\begin{prop} \label{pro:flambdas} Let $E$ be a vector bundle over $M$ of rank $k$. Let $s$ be a local section of $E$ which is a geometric monomial $($resp.\ polynomial$)$ with respect to the metric $g$ of Gilkey order ${\rm ord}^{\rm Gi}(s)$. The local section $\mathfrak{f}_\lambda^*s$ is a geometric monomial $($resp.\ polynomial$)$ with respect to the metric $g_\lambda$ of the same Gilkey order ${\rm ord}^{\rm Gi}(s)$.
\end{prop}
\begin{proof} Since the local section $s$ is a geometric monomial with respect to the metric $g$, it can be written as $s(\cdot,g)=\sum_{j=1}^k \alpha_j(\cdot,g)s_j(\cdot,g)$ where $\alpha_j(\cdot,g)$ is a monomial in the jets of the vielbeins. The section $\mathfrak{f}_\lambda^*s$ is equal to $\mathfrak{f}_\lambda^*s=\sum_{j=1}^k \alpha_j(\mathfrak{f}_\lambda(\cdot),g)s_j(\mathfrak{f}_\lambda(\cdot),g)$. In order to express $\mathfrak{f}_\lambda^*s$ in terms of the metric $g_\lambda$, we first differentiate the relations from Lemma~\ref{lem:vielbeincomp} in Appendix~\ref{sec:appendix1},
\[
a_{i}^{l}(\cdot,g_\lambda)=a_{i}^{l}(\mathfrak{f}_\lambda(\cdot),g) \qquad\text{and}\qquad b_{l}^{i}(\cdot,g_\lambda)=b_{l}^{i}( {\mathfrak f}_\lambda(\cdot),g),
\]
to get that $\alpha_j(\cdot,g_\lambda)=\lambda^{{\rm ord}^{\rm Gi}(s)}\alpha_j(\mathfrak{f}_\lambda(\cdot),g)$. On the other hand, by Lemma~\ref{ref:eq:patrgglambda}, we have that $\overline{s}_j(\cdot,g_\lambda)=s_j(\mathfrak{f}_\lambda(\cdot),g)$ for all $j$. Therefore, we deduce that
\[\mathfrak{f}_\lambda^*s=\lambda^{-{\rm ord}^{\rm Gi}(s)}\sum_{j=1}^k \alpha_j(\cdot,g_\lambda)\overline{s}_j(\cdot,g_\lambda).\]
That means $\mathfrak{f}_\lambda^*s$ is a geometric monomial with respect to the metric $g_\lambda$ and that $s$ and $\mathfrak{f}_\lambda^*s$ have the same Gilkey order.
\end{proof}

As a direct consequence, we get the following result on geometric differential operators.
\begin{prop} \label{pro:geometricopes} Let $E$ be a vector bundle over $M$ of finite rank equipped with an affine connection. Let $P$ be a differential operator in $ {\rm Diff}(M, E)$ of order $m$ geometric with respect to the metric $g$. The differential operator $\mathfrak{f}_\lambda^\sharp P$ in $ {\rm Diff}(M, E)$ of order $m$ is geometric with respect to the metric $g_\lambda$.
\end{prop}
\begin{proof} { If $P= \sum_{\vert \gamma\vert \leq m} P_\gamma D^\gamma$, then $\mathfrak{f}_\lambda^\sharp P= \sum_{\vert \gamma\vert \leq m} \mathfrak{f}_\lambda^* P_\gamma \mathfrak{f}_\lambda^\sharp D^\gamma$.
Since by assumption, the coefficients $ P_\gamma$ are geometric polynomials with respect to the metric $g$, it follows from the previous proposition that their pull-backs $\mathfrak{f}_\lambda^* P_\gamma $ are geometric polynomials with respect to the metric $g_\lambda$ with the same Gilkey order as those of the $P_\gamma$'s. Since the order of $\mathfrak{f}_\lambda^\sharp P$ coincides with that of~$P$, for any multi-index $\gamma$ we have
\[ {\rm ord}^{\rm Gi}(\mathfrak{f}_\lambda^* P_{\gamma})= {\rm ord}^{\rm Gi}(P_{\gamma})={\rm ord}(P)- |\gamma|={\rm ord}\bigl(\mathfrak{f}_\lambda^\sharp P\bigr)- |\gamma|,\]
so that $\mathfrak{f}_\lambda^\sharp P$ satisfies equation~(\ref{eq:geomopGi}).}
\end{proof}
 	
\section{Getzler rescaled geometric differential operators}

In this section, we focus on geometric differential operators $P$ in ${\rm Diff}(M,E)$ for $E=\Sigma M$, resp.\ $E=\Lambda T^*M$ given in Definition~\ref{defn:geometricdiffop}. We study their behaviour under the combined action of a~contraction ${\mathfrak f}_\lambda$ and the map $U_\lambda$ defined in \eqref{eq:actionulambda} as well as the limit as $\lambda$ tends to zero of the resulting operator. For that purpose we introduce the operators $\widetilde\Pp^{\rm Ge}|_{(\cdot,\lambda)}$ in \eqref{eq:PPlambdaGeLambda}, resp.\ in \eqref{eq:PPlambdaGeSigma}, where the superscript Ge stands for Getzler. We call $P$ rescalable if \smash{$\widetilde\Pp^{\rm Ge}|_{(\cdot,\lambda)}$} admits a limit~$\widetilde\Pp^{\lim} $ when $\lambda $ tends to zero. In Proposition~\ref{limitdiffope} in the case $E=\Lambda T^*M$, resp.\ in Proposition~\ref{prop:conditionrescalibilityClifford} in the case $E=\Sigma M$, we give a necessary and sufficient condition for $P$ to be rescalable and show that the coefficients of the limit operator $\widetilde\Pp^{\lim} $ are polynomials in the jets of the curvature tensor on~$M$. For $P$ in ${\rm Diff}(M,\Lambda T^*M)$, resp.\ in ${\rm Diff}(M,\Sigma M)$, we give in Theorem~\ref{thm:limPglambda}, resp.\ Corollary~\ref{cor:resclimClifford} (for an even order operator) a localisation formula similar to the one in Proposition~\ref{prop:reslim}. This time instead of the local residue form \smash{$\omega_{\log_\theta (\widetilde\Pp^{\lim}) }^{{ {\rm Res}}}$}, our localisation formula involves the local {$n$-form}~$\widetilde	 \omega_{\log_\theta (\widetilde\Pp^{\lim}) }^{{ {\rm Res}}}$.

In the same spirit as \eqref{eq:PPlambda} and \eqref{eq:PPlambdaflambda}, we set for $P\in {\rm Diff}(M, \Lambda T^* M)$ of order $m$
\begin{equation}\label{eq:PPlambdaGeLambda}
\Pp^{\rm Ge}\big|_{(\cdot,\lambda)} :=\lambda^mU_\lambda^\sharp P\big|_{(\cdot)} \qquad\text{and}\qquad \widetilde\Pp_\lambda^{\rm Ge}:=\mathfrak{f}_\lambda^\sharp \Pp^{\rm Ge},
\end{equation}
resp.\ for $P\in {\rm Diff}(M, \Sigma M)$
\begin{equation}\label{eq:PPlambdaGeSigma} \Pp^{\rm Ge}\big|_{(\cdot,\lambda)} :=\lambda^m \bigl(U_\lambda ^\sharp\circ c^g (P )\bigr)\big\vert_{(\cdot)} \qquad\text{and}\qquad \widetilde\Pp_\lambda^{\rm Ge}:=\mathfrak{f}_\lambda^\sharp \Pp^{\rm Ge},
\end{equation}
where $U_\lambda^\sharp P$ and $ c^g(P)$ are given in equations~\eqref{eq:ulambdap} and \eqref{eq:cgp}, respectively.

\begin{defn}
\label{defn:limitdiffope} We call a differential operator $P $ in ${\rm Diff} (M, E )$ with $E=\Lambda T^*M$, resp.\ $E=\Sigma M $ \emph{rescalable at a point $p$} if and only if \smash{$\widetilde \Pp_\lambda^{\rm Ge}\big\vert_{U_p}$} introduced in (\ref{eq:PPlambdaGeLambda}), {resp.\ in} (\ref{eq:PPlambdaGeSigma}) admits a limit~$\widetilde \Pp^{\lim}$ when $\lambda $ goes to zero.
\end{defn}

\begin{rk}
Note that rescalability is a local notion valid at a point, and is defined via a local normal geodesic coordinates.
\end{rk}
\begin{prop}
Rescalable operators in ${\rm Diff} (M, E)$ at point $p$ for $E=\Lambda T^*M$, resp.\ $E=\Sigma M $, form a subalgebra.
\end{prop}

\begin{proof}
Let $P_1$, $P_2$ be two operators in ${\rm Diff} (M, \Lambda T^* M)$ of order $m_1$ and $m_2$, respectively. Since the order is additive on products of operators and the degree is also additive on wedge products of forms, we have
\[\lambda^{m_1+m_2} {\mathfrak f}_\lambda^\sharp U_\lambda^\sharp(P_1P_2)=\bigl(\lambda^{m_1} {\mathfrak f}_\lambda^\sharp U_\lambda^\sharp(P_1) \bigr)\bigl(\lambda^{m_2} {\mathfrak f}_\lambda^\sharp U_\lambda^\sharp(P_2)\bigr). \]
If $P_1$ and $P_2$ are rescalable, the limits as $\lambda\to 0$ exist on the right-hand side, and hence so do they on the left-hand side, which shows that the product $P_1P_2$ is rescalable.
 Replacing $U_\lambda^\sharp $ by~$U_\lambda^\sharp\circ c^g $ and using the fact that $c^g$ is an algebra morphism yields the result for $E=\Sigma M$.
\end{proof}

The following technical lemma will be useful.

\begin{lem}\label{lem:lambdadeltaB}
	In local normal geodesic coordinates $X$, and with the notations of~\eqref{eq:valp}, let
$	q:={\rm val}_{X,p} (h)$ be the valuation of a local section $h\in \Gamma_p(E)$ around $p$, where $E$ is the trivial bundle $E:=M\times\R\to M$ and $p\in M$.
	\begin{enumerate}\itemsep=0pt
		\item[$1.$] For any real number $\theta$, as $\lambda$ tends to zero, the expression $ \lambda^{-\theta}D^\gamma(h\circ \mathfrak{f}_\lambda)$
		\begin{itemize}\itemsep=0pt
			\item converges if and only if $ \theta\leq \max (\vert\gamma\vert, q)$;
			\item if $\theta <\max (\vert\gamma\vert, q)$, it converges to zero.
		\end{itemize}
		\item[$2.$] If $\theta=\max (\vert\gamma\vert, q)$, the expression $ \lambda^{-\theta}D^\gamma(h\circ \mathfrak{f}_\lambda)$ converges to the coefficient of order $\theta-\vert \gamma\vert$ in the Taylor expansion of $h$ at point $p$.
	\end{enumerate}
\end{lem}
\begin{proof} The proof is based on the fact that $\lambda^{-\theta}D^\gamma(h\circ\mathfrak{f}_\lambda)=\lambda^{\vert\gamma\vert-\theta} (D^\gamma h)\circ \mathfrak{f}_\lambda$. By definition of~$q$, we have $h=O(\vert x\vert^q)$ so that for $q\geq \vert\gamma\vert$, we have $ (D^\gamma h)\circ \mathfrak{f}_\lambda =\lambda^{q-\vert\gamma\vert}O\bigl(\vert x\vert^{q-\vert\gamma\vert}\bigr)$. If, $q< \vert\gamma\vert$, we have $(D^\gamma h)\circ \mathfrak{f}_\lambda =O(1),$ which ends the proof of (1) and (2) observing that in the convergent case, the limit corresponds to the $\theta-\vert \gamma \vert$ coefficient in the Taylor expansion.
\end{proof}

 We now specialise to a monomial {say $P$} in the jets of the vielbeins as in \eqref{eq:polynomialsvielbeins} written in normal geodesic coordinates. We set
 \begin{equation*} 
 \Theta^{X, p}_{(x, g)}{(P)}:={\sum_s \max \bigl(|\beta_s|, {\rm val}_{X,p}\bigl(a_{i_s}^{t_s}(x,g)\bigr)\bigr)+ \max \bigl(|\gamma_s|, {\rm val}_{X,p}\bigl(b_{n_s}^{l_s}(x,g) \bigr)\bigr)}.
\end{equation*}
{\it In the following, we shall often drop the explicit mention of $X$, $p$, $x$, $g$ and simply write $\Theta (P)$.}

\begin{ex} \label{eq:examples}\quad
	\begin{enumerate}\itemsep=0pt
		\item
		Recall that $g_{ij}(x)=\sum_l a_i^l(x,g)a_j^l(x,g)$. We have $\Theta(g_{ij})=0$ {if $i=j$ } and it is at least $2$ otherwise.
		\item Using the Koszul formula (\ref{eq:Koszul}), the Christoffel symbols $\Gamma_{ij}^k(x,g)$ can be written
		$\Gamma_{ij}^k(x,g)=\sum_l P_{i,j,k}^l(x,g)$
		with ${\Theta\bigl({P_{i,j,k}^l}\bigr)}=1$ for $l=i=j=k$ and at least $2$ otherwise.
		\item Similarly, by relation \eqref{eq:relationschristofellsymbols} the Christoffel symbols $\tilde\Gamma_{ls}^t(x,g)$ can be written
		\[\tilde\Gamma_{ls}^t(x,g)=\sum_k P_{l,s,t}^k(x,g)\]
		with ${\Theta\bigl({P_{l,s,t}^k}\bigr)}\geq 2$. Notice here that ${\Theta\bigl({P_{l,s,t}^k}\bigr)}$ cannot be equal to $1$, since this corresponds to $l=s=t=k$ which would imply $\tilde\Gamma_{ls}^t(x,g)=\tilde\Gamma_{ls}^s(x,g)=0$ and hence would yield a~contradiction.
	\end{enumerate}
\end{ex}

\begin{prop} \label{limitdiffope}
Let $P\in{\rm Diff} (M,\Lambda T^* M)$ be a geometric differential operator of order $m$ with respect to a metric $g$. In a local trivialisation around a point $p$ of $\Lambda T^*M$ induced by normal geodesic coordinates $\bigl(x^1, \dots, x^n\bigr)$ on $U_p$, the operator $P$ applied to a section $s=\sum_{I} \alpha_I {\rm d}x^I$ $\bigl({\rm d}x^I:= {\rm d}x^{i_1}\wedge \dots \wedge {\rm d}x^{i_{\vert I\vert}}$ for $I=\{i_1, \dots, i_{\vert I\vert}\}\bigr)$ reads
 \begin{equation*}
 P \biggl(\sum_{I\subset \N} \alpha_I {\rm d}x^I \biggr)=\sum_{\vert \gamma\vert \leq m}\sum_{I,J\subset \N} (P_\gamma )_{I J}(\cdot,g) D^\gamma (\alpha_I) {\rm d}x^J,
 \end{equation*}
 where $(P_\gamma)_{IJ}(\cdot,g)$ are polynomials as in \eqref{eq:polynomialsvielbeins}. The operator $P$ is rescalable at a point $p$ if and only if $\vert J\vert -\vert I\vert\leq \Theta((P_\gamma)_{IJ})$.
In this case, the limit rescaled operator reads
\[
	 \widetilde \Pp^{\lim} =
 \sum_{\vert \gamma\vert \leq m}\sum_{\vert J\vert -\vert I\vert =\Theta({(P_\gamma)_{IJ}})} (P_\gamma)_{I J}^{\lim}(p,g) \bigl(\bigl({\rm d}x^I\bigr)^*\otimes {\rm d}x^J\bigr) D^{\gamma},
\]
where
\[
 (P_{\gamma})_{IJ}^{\lim}(p, g)=
	\mathop{\lim}\limits_{\lambda\to 0}\bigl(\lambda^{|I|-\vert J\vert}(P_{\gamma})_{IJ}(\cdot, g_\lambda)\bigr),\qquad {\forall I, J\subset \N}
\]
is a polynomial expression in the jets of the Riemannian curvature tensor.
\end{prop}

\begin{proof} The local expression of $P$ in the theorem results from Remark~\ref{lem:subbundlemetric}, equations~\eqref{eq:PgLambdaT*Mcoord} and~\eqref{eq:tensor}. Hence, from the definition of $U_\lambda^\sharp P=U_\lambda P U_\lambda^{-1}$, we get
\[
\bigl(U_\lambda^\sharp P\bigr)\biggl(\sum_{I\subset \N} \alpha_I {\rm d}x^I\biggr)=\sum_{\vert \gamma\vert \leq m}\sum_{I,J\subset \N}\lambda^{|I|-|J|} (P_\gamma)_{I J}(\cdot,g) D^\gamma (\alpha_I) {\rm d}x^J.
\]
Since $P$ is a geometric differential operator, the coefficients $(P_\gamma)_{IJ}(\cdot,g)$ are in the jets of the vielbeins as in \eqref{eq:polynomialsvielbeins} with ${\rm ord}^{\rm Gi}((P_\gamma)_{IJ})+|\gamma|=m$. Now, we write
\begin{align}
 \widetilde \Pp_\lambda^{\rm Ge}(s\circ\mathfrak{f}_\lambda)& =\lambda^m\bigl(\mathfrak{f}_\lambda^\sharp U_\lambda^\sharp P\bigr)(s\circ \mathfrak{f}_\lambda)
=\lambda^m\bigl(U_\lambda^\sharp P\bigr)(s)\circ \mathfrak{f}_\lambda\nonumber\\
& = \sum_{\vert \gamma\vert \leq m}\sum_{I,J\subset \N} \lambda^{ m+\vert I\vert-\vert J\vert} (P_\gamma)_{IJ}(\mathfrak{f}_\lambda(\cdot),g) D^\gamma (\alpha_I) {\rm d}x^J\big|_{\mathfrak{f}_\lambda(\cdot)}\nonumber\\
& =\sum_{\vert \gamma\vert \leq m}\sum_{I,J\subset \N} \lambda^{\vert\gamma\vert+\vert I\vert-\vert J\vert} (P_\gamma)_{IJ}(\cdot,g_\lambda) \lambda^{-\vert\gamma\vert}D^\gamma (\alpha_I\circ \mathfrak{f}_\lambda) {\rm d}x^J\circ\mathfrak{f}_\lambda.\label{eq:PlambdaGeg}
\end{align}
In the last equality, we use the fact that $(P_\gamma)_{IJ}(\cdot,g_\lambda)=\lambda^{{\rm ord}^{\rm Gi}((P_\gamma)_{IJ})}(P_\gamma)_{IJ}(\mathfrak{f}_\lambda(\cdot),g)$. Hence, we deduce that
\[
\widetilde\Pp_\lambda^{\rm Ge}=\sum_{\vert \gamma\vert \leq m}\sum_{I,J\subset \N} \lambda^{\vert I\vert-\vert J\vert} (P_\gamma)_{IJ}(\cdot,g_\lambda) \bigl(\bigl({\rm d}x^I\bigr)^*\otimes {\rm d}x^J\bigr)\big|_{\mathfrak{f}_\lambda} D^{\gamma}.
\]
Now by Lemma~\ref{lem:vielbeincomp} in Appendix~\ref{sec:appendix1}, we write for $(P_\gamma)_{IJ}(\cdot,g_\lambda)$
\[
 \prod_{s=1}^S D^{\alpha_s}\bigl(a_{i_s}^{t_s}(\cdot ,g_\lambda)\bigr) D^{\beta_s}\bigl(b^{l_s}_{n_s}(\cdot ,g_\lambda)\bigr)= \prod_{s=1}^S D^{\alpha_s}\bigl(a_{i_s}^{t_s}(\cdot ,g )\circ \mathfrak{f}_\lambda\bigr) D^{\beta_s}\bigl(b^{l_s}_{n_s}(\cdot ,g )\circ \mathfrak{f}_\lambda\bigr).
\]
 For convenience, we have dropped the explicit mention of the indices $I$ and $J$.
 Applying Lemma~\ref{lem:lambdadeltaB} to $h_a:= a_{i_s}^{t_s}(\cdot ,g )$ and $h_b:= b^{l_s}_{n_s}(\cdot ,g )$ with both $\theta_a$ and $\theta_b$ non negative integers such that $\theta_a+\theta_b=\vert J\vert -\vert I\vert$, it tells us that the expression $ \lambda^{-\theta_a}D^\gamma( \mathfrak{f}_\lambda^* h_a)$ converges if and only if $\theta_a\leq \max (\vert\alpha_s\vert, q_a)$, with $q_a:= {\rm val}_{X,p}\bigl(a_{i_s}^{t_s}(x,g)\bigr)$ and that the limit vanishes if we have a strict inequality. If $\theta_a= \max (\vert\alpha_s\vert, q_a)$, the limit is a polynomial in the jets of the curvature tensor. Similarly for $h_b$ and $\theta_b$. Hence, the only non zero terms which survive in the limit of (\ref{eq:PlambdaGeg}) as $\lambda \to 0$, correspond to $\theta_a=\max (\vert\alpha_s\vert, q_a)$ and $\theta_b= \max (\vert\beta_s\vert, q_b)$ and hence $\Theta({(P_{\gamma})_{IJ}})=\vert J\vert -\vert I\vert$. This yields the statement of the theorem.
\end{proof}

We prove a localisation formula for the local form $\widetilde	 \omega_{\log_\theta (P)}^{{ {\rm Res}}}$ when $E=\Lambda T^*M$.

\begin{thm} \label{thm:limPglambda}
Let $P$ in ${\rm Diff} (M,\Lambda T^* M)$ be a geometric differential operator with respect to the metric $g$ of Agmon angle $\theta$. If $P$ is rescalable, then \[\widetilde	 \omega_{\log_\theta (P)}^{{ {\rm Res}}}(p)= \widetilde\omega_{\log_\theta (\widetilde \Pp^{\lim})}^{{ {\rm Res}}}(x), \qquad \forall x\in U_p,\]
where $\widetilde \Pp^{\lim}:=\lim_{\lambda \to 0} \widetilde \Pp_\lambda^{\rm Ge}$.
\end{thm}
\begin{proof}
Combining equation~(\ref{eq:PPlambdaGeLambda}) with equation~(\ref{eq:tildeomegaflambda}) yields
\[\tilde \omega^{\rm Res}_{\log_\theta (P)}( {\mathfrak f}_{ \lambda}(x))=\tilde \omega^{\rm Res}_{\log_\theta (U_\lambda^\sharp({\mathfrak f}_{\lambda}^\sharp P))}(x)=\widetilde	 \omega_{\log_\theta (\widetilde \Pp^{\rm Ge}_\lambda)}^{{ {\rm Res}}}(x),\]
where we have commuted $U_\lambda^\sharp$ and $\mathfrak f_\lambda^\sharp$ thanks to equation~(\ref{eq:flambdaUlambda}).
By Proposition~\ref{limitdiffope}, the limit $\widetilde \Pp^{\lim} :=\lim_{\lambda \to 0} \widetilde \Pp_\lambda^{\rm Ge}$ exists, from which we deduce the statement of the theorem by letting $\lambda$ tend to zero in the above identities.
\end{proof}

\begin{ex} \label{ex:Hodgelaplacian} The exterior differential ${\rm d}\colon\Lambda T^*M\to \Lambda T^*M$ \big(as well as its $L^2$-adjoint $\delta$\big) is not~a~rescalable operator. Indeed, by writing ${\rm d}=\sum_{j=1}^n {\rm d}x^j\wedge\smash{\nabla_{\frac{\partial}{\partial x^j}}}$ and using the local expression of~$\nabla_{\frac{\partial}{\partial x^j}}$ in Example~\ref{ex:nablaXdiffgeom}, we have for any $\alpha=\sum_I\alpha_I {\rm d}x^I$ that
\begin{gather*}
 {\rm d}\biggl(\sum_I\alpha_I {\rm d}x^I\biggr)=\sum_{j,I}\biggl({\frac{\partial}{\partial x^j}}\alpha_{I}\biggr){\rm d}x^j\wedge {\rm d}x^I\\
 \qquad {} +\sum_{t,I}\alpha_{I}\Biggl(\sum_{s,l=1}^n g\bigl(\nabla_{\frac{\partial}{\partial x^j}} {\rm d}x^{i_s},{\rm d}x^l\bigr)g_{tl}(\cdot)\Biggr){\rm d}x^j\wedge {\rm d}x^{i_1}\wedge\dots\wedge \underbrace{{\rm d}x^t}_{s^{\rm th}\text{{\rm -slot}}}\wedge\dots\wedge {\rm d}x^{i_k},
\end{gather*}
which shows that it is a geometric operator. For the $j$'s that do not belong to $I$ in the first sum of the right-hand side, we have $|J|-|I|=1$ and the corresponding $\Theta=0$. Therefore, the condition in Proposition~\ref{limitdiffope} is not fulfilled and, thus, ${\rm d}$ is not rescalable. However, the Hodge operator $\Delta={\rm d}\delta+\delta {\rm d}$ is a geometric rescalable operator. The fact that it is geometric comes from Proposition~\ref{prop:productgeom}. To show that it is rescalable, we use the Bochner--Weitzenb\"ock formula on $k$-forms: $\Delta=\nabla^*\nabla+W^{[k]}$, where $\nabla^*\nabla$ is given by
\[
\nabla^*\nabla=-g^{ij}(x)\bigl(\nabla_{\frac{\partial}{\partial x^i}}\nabla_{\frac{\partial}{\partial x^j}}-\Gamma_{ij}^k(x, g)\nabla_{\frac{\partial}{\partial x^k}}\bigr)
\]
and $W^{[k]}=\sum_{i,j=1}^n e_j^*\wedge (e_i\lrcorner R(e_i,e_j))$ is the Bochner operator. Here $R$ is the curvature operator of the manifold $M$. Indeed, by replacing $\nabla_{\frac{\partial}{\partial x^i}}$ by its expression and performing some computations, one can easily see that $\Delta$ (we use Einstein convention) has the form
\begin{align*}
\Delta\bigl(\alpha_I {\rm d}x^I\bigr)={}&-g^{ij}\frac{\partial}{\partial x^i}\frac{\partial}{\partial x^j}(\alpha_I) {\rm d}x^I\\
& {}-g^{ij}\frac{\partial}{\partial x^j}(\alpha_I)\Gamma_{ii_s}^k g_{tk} {\rm d}x^{i_1}\wedge\dots\wedge \underbrace{{\rm d}x^t}_{s^{\rm th}\text{{\rm -slot}}}\wedge\dots\wedge {\rm d}x^{i_k} +\cdots.
\end{align*}
Since the Laplacian preserves the degree, then $|J|-|I|=0\leq \Theta$ is always satisfied. The limit of the rescaled operator corresponds to polynomials with $\Theta=0$. Hence by the computation of~$\Theta$ in Example~\ref{eq:examples}, we get that
\[\widetilde \Pp^{\lim}\biggl(\sum_I \alpha_I {\rm d}x^I\biggr)=-\sum_{i,I}\frac{\partial^2\alpha_I}{(\partial x^i)^2} {\rm d}x^I.\]
\end{ex}

Therefore, the localisation formula in Theorem~\ref{thm:limPglambda} can be applied for the Hodge Laplacian and we get \[\widetilde	 \omega_{\log_\theta (\Delta)}^{{ {\rm Res}}}(p)= \widetilde\omega_{\log_\theta (\widetilde \Pp^{\lim})}^{{ {\rm Res}}}(x), \qquad \forall x\in U_p.\]

\begin{prop}\label{prop:conditionrescalibilityClifford}
 Let $P\in {\rm Diff}(M,\Sigma M)$ be a geometric differential operator with respect to the metric $g$ of order $m$. In the trivialisation $\bigl\{{\rm e}^1,\dots, {\rm e}^n\bigr\}$ induced by parallel transport, the operator reads
 \begin{equation}\label{eq:geometricoperatorspin}
 P=\sum_{|\gamma|\leq m}\sum_I (P_\gamma)_I(\cdot,g){\rm e}^{I}\cdot_g D^\gamma,
 \end{equation}
 where $(P_\gamma)_I(\cdot,g)$ are polynomials as in \eqref{eq:polynomialsvielbeins} and ${\rm e}^I\cdot_g={\rm e}^{i_1}\cdot_g\cdots \cdot_g {\rm e}^{i_k}$ with $i_1<\dots<i_k$ and {$|I|= k$}. The operator $P$ is rescalable if and only if $|I|\leq \Theta({(P_\gamma)_{I}}) $. In this case, the limit rescaled operator in \eqref{eq:PPlambdaGeSigma} reads
\[ {\widetilde \Pp}^{\lim}
 = \sum_{\vert \gamma\vert \leq m}\sum_{\vert I\vert =\Theta({(P_\gamma)_{I}}) } (P_\gamma)_{I}^{\lim}(p,g) {\rm e}^I\wedge D^\gamma,\]
where $(P_{\gamma})_{I}^{\lim}(p, g):=
\lim_{\lambda\to 0}\bigl(\lambda^{-|I|}(P_{\gamma})_{I}(\cdot, g_\lambda)\bigr)$
is a polynomial expression in the jets of the Riemannian curvature tensor.
\end{prop}

\begin{proof} In any local trivialization of $\Sigma M$, the operator $P$ can be written as
\[
P=\sum_{\vert \gamma\vert \leq m} P_\gamma (\cdot,g) D^\gamma,
\]
 where $P_\gamma(\cdot,g)\in \operatorname{End}(\Sigma M)\simeq \Cl(T M)\otimes \mathbb{C}$. In the local trivialisation above a normal geodesic chart induced by parallel transport, we write
\[
P_\gamma (\cdot,g)=\sum_{I} (P_\gamma)_I(\cdot,g){\rm e}^{I}.
\]
Recall here that the $\operatorname{End}(\Sigma_p M)$-valued functions ${\rm e}^i \cdot_g\colon \sigma \mapsto {\rm e}^i \cdot_g \sigma$ for each $i=1, \dots, n$ on $U_p$ are constant along the geodesics and hence in this trivialisation \cite[Lemma 4.14]{BGV}. Now, as $P$ is geometric, we get that $(P_\gamma)_I(\cdot,g)$ are in the jets of the vielbeins as in \eqref{eq:polynomialsvielbeins}. Hence, we get \eqref{eq:geometricoperatorspin}. In particular, we deduce that
\begin{equation*}
 c^g(P)=\sum_{|\gamma|\leq m}\sum_I (P_\gamma)_I(\cdot,g)c^g\bigl({\rm e}^{I}\bigr) D^\gamma.
 \end{equation*}
Now, we apply equation~\eqref{eq:PPlambdaGeSigma} to a section $s$ to get
\begin{align*}
\widetilde \Pp_\lambda^{\rm Ge}(s\circ\mathfrak{f}_\lambda) &{}=\lambda^m\bigl(\mathfrak{f}_\lambda^\sharp U_\lambda^\sharp c^g(P)\bigr)(s\circ \mathfrak{f}_\lambda)=\lambda^m\bigl(U_\lambda^\sharp c^g(P)s\bigr)\circ\mathfrak{f}_\lambda\nonumber\\
&{}= \sum_{\vert \gamma\vert \leq m}\sum_{I\subset \N} \lambda^m(P_\gamma)_{I}(\mathfrak{f}_\lambda(\cdot),g) U_\lambda^\sharp \bigl(c^g\bigl({\rm e}^I\bigr)\bigr)(D^\gamma s)\big|_{\mathfrak{f}_\lambda(\cdot)}.\nonumber\\
&{}=\sum_{\vert \gamma\vert \leq m}\sum_{I\subset \N} \lambda^{m-|I|}(P_\gamma)_{I}(\mathfrak{f}_\lambda(\cdot),g) \lambda^{|I|}U_\lambda^\sharp \bigl(c^g\bigl({\rm e}^I\bigr)\bigr)\lambda^{-\vert\gamma\vert}D^\gamma (s\circ\mathfrak{f}_\lambda)\nonumber\\
&{}=\sum_{\vert \gamma\vert \leq m}\sum_{I\subset \N} \lambda^{-|I|}(P_\gamma)_{I}(\cdot,g_\lambda) \lambda^{|I|}U_\lambda^\sharp \bigl(c^g\bigl({\rm e}^I\bigr)\bigr)D^\gamma (s\circ\mathfrak{f}_\lambda).
\end{align*}
Here, we use the fact that
\[
(P_\gamma)_{I}(\cdot,g_\lambda)=\lambda^{{\rm ord}^{\rm Gi}((P_\gamma)_{I})}(P_\gamma)_{I}(\mathfrak{f}_\lambda(\cdot),g) \qquad \text{and}
\qquad {\rm ord}^{\rm Gi}((P_\gamma)_{I})+\vert\gamma\vert=m,
\]
 since $(P_\gamma)_{I}(\cdot,g)$ are in the jets of the vielbeins. Therefore, we deduce that
\[\widetilde\Pp_\lambda^{\rm Ge}=\sum_{\vert \gamma\vert \leq m}\sum_{I\subset \N} \lambda^{-|I|}(P_\gamma)_{I}(\cdot,g_\lambda) \lambda^{|I|}U_\lambda^\sharp \bigl(c^g\bigl({\rm e}^I\bigr)\bigr)D^\gamma.\]
Now, using \eqref{eq:limUlambdac}, we have that $\lambda^{|I|}U_\lambda^\sharp\bigl(c^g\bigl({\rm e}^{I}\bigr)\bigr)$ converges to ${\rm e}^I\wedge$ as $\lambda\to 0$. Also, by Lemma~\ref{lem:lambdadeltaB}, the term $\lambda^{-\vert I\vert}(P_\gamma)_I(\cdot,g_\lambda)$ converges if and only if $|I|\leq \Theta({(P_\gamma)_{I}}) $. Thus, the operator $P$ is rescalable if and only if $|I|\leq \Theta((P_\gamma)_I)$. The limit rescaled operator follows then easily.
\end{proof}

We prove a localisation formula for the local form
	$\widetilde	 \omega_{\log_\theta (P)}^{{ {\rm Res}}}$ when $E=\Sigma M$.
\begin{cor}\label{cor:resclimClifford} Let $P$ in ${\rm Diff}(M,\Sigma M)$ be a geometric differential operator with respect to the metric $g$ of Agmon angle $\theta$ which is even for the $\Z_2$-grading $\Sigma M=\Sigma^+ M\oplus \Sigma^- M$. If $P$ is rescalable, then we have
	\begin{equation*} \widetilde\omega_{\log_\theta (\widetilde \Pp^{\lim})}^{{ {\rm Res}}}(x)=(-2{\rm i})^{-n/2}\omega_{\log_\theta (P)}^{{ {\rm sRes}}}(p), \qquad \forall x\in U_p,\end{equation*}
	where $\widetilde \Pp^{\lim}:=\lim_{\lambda \to 0} \widetilde\Pp_\lambda^{\rm Ge}$ with $\widetilde\Pp_\lambda^{\rm Ge}$ as in \eqref{eq:PPlambdaGeSigma}.
\end{cor}

\begin{proof}
Using the relation in Corollary~\ref{cor:relationomegatilde}, the fact that $P$ is rescalable and that $j_g\circ\mathfrak{f}_\lambda\to 1$ as $\lambda\to 0$ yield the result.
\end{proof}

\section{The rescaled square of the Dirac operator}

In this section, we show that whereas the Dirac operator $\slashed D$ (which is a geometric operator and hence so its square) is not rescalable, its square is. We then apply the results of the previous section to {$P= \slashed D^2$} and compute $\widetilde \Pp^{\lim}$ with the help of Proposition~\ref{prop:conditionrescalibilityClifford}. This allows to find the expression of $\widetilde \Pp^{\lim}$ in terms of the curvature operator of $M$ as in \cite{Ge}. We then derive from Corollary~\ref{cor:resclimClifford} a localisation formula (\ref{eq:resD2}) for the graded residue of the logarithm of $\slashed D^2$.

We recall that the {\it Dirac operator } on a spin manifold $(M^n,g)$ is the differential operator of order one given by
$\slashed D:= \sum_{i=1}^n {\rm e}^i\cdot_g\nabla_{e_i}$,
	where $\nabla_{e_i}$ is the spinorial Levi-Civita covariant derivative in the direction of $e_i$. Using \eqref{eq:connspin}, it reads as
 \[
 \slashed 	D=\sum_{i=l} a_i^l(\cdot,g) {\rm e}^i\cdot_g \partial_{x^m}+\sum_{i\neq l} a_i^l(\cdot,g) {\rm e}^i\cdot_g \partial_{x^l} +\frac{1}{4}\sum_{i,l,s,t=1}^n\tilde\Gamma_{l s}^t(\cdot,g) a_i^l(\cdot,g) {\rm e}^i\cdot_g {\rm e}^s\cdot_g {\rm e}^t\cdot_g.
 \]
The Dirac operator is geometric. Indeed, the above expression involves a sum of three terms, each of which is expressed in terms of jets of vielbeins \big(see Example~\ref{ex:geompolTM} for $\widetilde \Gamma$\big) and satisfies condition (\ref{eq:geomopGi}). Yet it is not rescalable. With the notations of Proposition~\ref{prop:conditionrescalibilityClifford} with $P=\slashed D$, the condition $|I |\leq \Theta((P_\gamma)_I )$ is not satisfied in the first of the three sums since $|I|=1$ and $\Theta((P_\gamma)_I)=0$. Recall that the valuation of $a_i^l(\cdot,g)$ is zero if $l=i$ and at least $2$ otherwise.

\begin{prop}\label{prop:Dg2Geztlerrescalable} The square of the Dirac operator $\slashed 	D^2$ is a rescalable geometric differential operator. Setting $P:= \slashed D^2$, the operator $\widetilde \Pp_\lambda^{\rm Ge}$ in \eqref{eq:PPlambdaGeSigma}, read in a local trivialisation of~$\Sigma M$ at point $x$ in~$U_p$ obtained by parallel transport along geodesics, converges to
\[
- \sum_i\biggl( \partial_{x^i}-\frac{1}{8}\sum_{j,s,t} R_{i j st}(p) x^j {\rm e}^s\wedge {\rm e}^t\wedge\biggr)^2,
\]
where $\{e_1(p),\dots,e_n(p)\}$ is an orthonormal frame $T_pM$.
\end{prop}

\begin{proof} Since $\slashed D$ is geometric so is its square $\slashed D^2$ by Proposition~\ref{prop:productgeom}. We now show it is rescalable. Since the action by Clifford multiplication ${\rm e}^i\cdot_g$ of the vectors of an orthonormal frame of $T_xM$ obtained from $\{e_1(p),\dots,e_n(p)\}$ by parallel transport is constant in $x$, in the following we will simply write $\{e_1 ,\dots,e_n \}$. We use the Schr\"odinger--Lichnerowicz formula \cite{LM} to write
 \begin{align*}
 \slashed 	D^2\big|_x={}&-{\sum_{i,j=1}^n}g^{ij}(x)\bigl(\nabla_{x^i}\nabla_{x^j}-\Gamma_{ij}^k(x, g)\nabla_{x^k}\bigr)+\frac{1}{4}{\rm Scal}(x)\\
 ={}&-{\sum_{i,j=1}^n}{\sum_{k,l=1}^n}g^{ij}(x)\biggl(\partial_{x^i}+\frac{1}{4}\tilde\Gamma_{i k}^l(x,g) {\rm e}^k\cdot_g {\rm e}^l\cdot_g\biggr) {\sum_{s,t=1}^n}\left(\partial_{x^j}+\frac{1}{4}\tilde\Gamma_{j s}^{t}(x,g) {\rm e}^{s}\cdot_g {\rm e}^{t}\cdot_g\right)\\
 &{} +{\sum_{i,j, k=1}^n} g^{ij}(x)\Gamma_{ij}^k(x,g){\sum_{s,t=1}^n}\biggl(\partial_{x^k}+\frac{1}{4}\tilde\Gamma_{k s}^{t}(x, g){\rm e}^s\cdot_g {\rm e}^t\cdot_g\biggr)+\frac{1}{4}{\rm Scal}(x)\\
 ={}&\underbrace{-{\sum_{i,j=1}^n}g^{ij}(x)\partial^2_{x^i x^j}}_{\rm (I)}-\underbrace{{\sum_{i,j=1}^n}\frac{1}{4}g^{ij}(x){\sum_{s, t=1}^n}\partial_{x^i}\bigl(\tilde\Gamma_{j s}^{t}(x,g)\bigr){\rm e}^s\cdot_g {\rm e}^{t}\cdot_g}_{\rm (II)}\\
 &{}-\underbrace{{\sum_{i,j, s,t=1}^n}\frac{1}{2}g^{ij}(x)\tilde\Gamma_{js}^{t}(x,g) {\rm e}^{s}\cdot_g {\rm e}^{t}\cdot_g\partial_{x^i}}_{\rm (III)}\\
 &{} \underbrace{-\frac{1}{16}{\sum_{i,j=1}^n}\sum_{k\neq l, s\neq t} g^{ij}(x)\tilde\Gamma_{i k}^l(x,g)\tilde\Gamma_{j s}^{t}(x,g){\rm e}^k\cdot_g {\rm e}^l\cdot_g {\rm e}^{s}\cdot_g {\rm e}^{t}\cdot_g}_{\rm (IV)}\\
 &{}+\underbrace{{\sum_{i,j, k=1}^n}g^{ij}(x)\Gamma_{ij}^k(x,g)\Biggl(\partial_{x^k}+\frac{1}{4}{\sum_{s,t=1}^n}\tilde\Gamma_{k s}^{t}(x, g){\rm e}^s\cdot_g {\rm e}^t\cdot_g\Biggr)}_{\rm (V)}+\frac{1}{4}{\rm Scal}(x),
 \end{align*}
 where ${\rm Scal}$ is the scalar curvature of the metric $g$. Combining equations~(\ref{eq:AAt}) and (\ref{eq:relationschristofellsymbols}) with the Koszul formula, we can express $\slashed D^2$ in terms of the vielbeins. To avoid lengthy computations, we only sketch the computation for the third term in the above equation to show that the relation $|I|\leq \Theta(({P_\gamma})_I)$ holds. Thus, the operator is rescalable. To show the inequality, we first observe that $|I|$ equals $2$. According to Examples~\ref{eq:examples}, the coefficient in (III) can be written as a sum of polynomials of the vielbeins such that the corresponding $\Theta$ is at least $2$ \big(since the one corresponding to $\tilde{\Gamma}_{i m}^l(x,g)$ is at least $2$\big). With the help of \eqref{eq:symbchriglamdabis} and \eqref{eq:Gammatilde}, the limit of the rescaled operator of (III) is equal to
 \[\mathop{\lim}\limits_{\lambda\to 0} \lambda^{-2}\sum_{i,j,s,t} g_\lambda^{ij}(x)\tilde\Gamma_{j s}^t(x,g_\lambda)\bigl({\rm e}^s\wedge {\rm e}^t\wedge\bigr)\partial_{x^i}=-\frac{1}{2}\sum_{i,j,k,s,t}\delta^{ij}R_{j k s t}( p)x^k\bigl({\rm e}^s\wedge {\rm e}^t\wedge\bigr)\partial_{x^i}.\]
The same thing can be done for the first term (I) which converges to $-\sum_i \partial^2_{x^i}$. The second term~(II) converges to
$\frac{1}{8}\sum_{i,j,s,t}\delta^{ij} R_{ijst}(p){\rm e}^s \wedge {\rm e}^{t}\wedge.$
The fourth term (IV) tends to
\[-\frac{1}{64}\sum_{i,j,k,l,s,t,q} \delta^{ij} R_{iq k l}(p) R_{jl s t}(p)x^q x^{l}{\rm e}^k\wedge {\rm e}^l\wedge {\rm e}^{s}\wedge {\rm e}^{t}\wedge.\]
The other terms converge to $0$. Therefore, we deduce that
\begin{align*}
 \mathop{\lim}\limits_{\lambda\to 0} \widetilde \Pp_\lambda^{\rm Ge}={}&-\sum_i \partial^2_{x^i}+\frac{1}{8}\sum_{i,j,s,t}\delta^{ij} R_{ijst}(p){\rm e}^s \wedge {\rm e}^{t} +\frac{1}{4}\sum_{i,j,k,s,t}\delta^{ij}R_{j k s t}( p)x^k\bigl({\rm e}^s\wedge {\rm e}^t\wedge\bigr)\partial_{x^i}\\
 &{}-\frac{1}{64}\sum_{i,j,k,l,s,t,q} \delta^{ij} R_{iq k l}(p) R_{jl st}(p)x^q x^{l}{\rm e}^k\wedge {\rm e}^l\wedge {\rm e}^{s}\wedge {\rm e}^{t}\wedge\\
 ={}&-\sum_i\biggl(\partial_{x^i}-\frac{1}{8}\sum_{j,s,t}R_{i j st}(p) x^j {\rm e}^s\wedge {\rm e}^t\wedge\biggr)^2,
\end{align*}
which confirms the fact, which we checked by hand in the previous tedious computations, that the operator is rescalable.
 \end{proof}
		
We can apply Corollary~\ref{cor:resclimClifford} to the square of the Dirac operator $P=\slashed D^2$. Since the operator~$\slashed D^2$ is non negative self-adjoint, it has a well defined logarithm $\log_\theta \bigl(\slashed 	D^2\bigr)$ (here $\theta=\pi$).
Hence, we have that	
\begin{equation}
\label{eq:resD2} 	
\widetilde\omega_{\log_\theta (\widetilde \Pp^{\lim})}^{{ {\rm Res}}}(x)= (-2{\rm i})^{-n/2} 	\omega_{\log_\theta (\slashed D^2) }^{\rm sRes}(p),
\end{equation}
for any $x\in U_p$. The computation of the Wodzicki residue on the right-hand side is tedious. In~\cite{Sc}, it is derived from the heat-kernel asymptotics of $\widetilde \Pp^{\lim}$.

\appendix
\section{Vielbein} \label{sec:appendix1}

Let $(M^n,g)$ be a Riemannian manifold. Let $F_p(\cdot,g):=\big\{\frac{\partial}{\partial x^1},\dots, \frac{\partial}{\partial x^n}\big\}$ be the cartesian frame on $U_p$ built from the geodesic coordinates around $p$ defined in (\ref{eq:geodesiccoord}). From a given orthonormal basis $e_1(p) ,\dots,e_n(p) $ of $T_pM$ at $p\in M$, we build a local orthonormal frame $O_p(\cdot, g):=\{e_1(\cdot, g),\dots,e_n(\cdot, g)\},$ of $TM$ obtained by the parallel transport along small geodesics as in~(\ref{eq:Op}).

A linear map $A_p(\cdot, g)\colon TM\to TM$ (resp.\ its inverse $B_p(\cdot,g)$) which takes the basis $F_p(\cdot, g)$ to $O_p(\cdot, g)$ (resp.\ $O_p(\cdot, g)$ to $F_p(\cdot, g)$) can be represented by a $n\times n$ matrix $A=\bigl( a_i^l(\cdot, g)\bigr)$ (resp.\ $B=\bigl(b_l^i(\cdot,g)\bigr) $) with
\begin{gather}
\frac{ \partial}{\partial x^i}=\sum_{l=1}^n a_i^l (\cdot, g) e_l(\cdot,g) ,\qquad \forall i\in [[1, n]],\nonumber\\
 e_l(\cdot,g)= \sum_{j=1}^nb_l^j(\cdot, g) \frac{\partial}{\partial x^j},\qquad \forall l\in [[1, n]].\label{eq:viedbein}
\end{gather}
 Also, we have that
\begin{equation*}
{\rm d} x^i=\sum_{m=1}^n b_l^i(\cdot, g) {\rm e}^l(\cdot,g) ,\qquad \forall i\in [[1, n]],\qquad {\rm e}^l(\cdot,g)= \sum_{j=1}^na_j^l(\cdot, g) {\rm d} x^j,\qquad \forall l\in [[1, n]].
\end{equation*}
With these conventions, and dropping the explicit mention of $p$ whenever this does not lead to confusion, we write $A(\cdot, g)=\bigl(a_i^l(\cdot, g)\bigr)_{i,l}$ and we have that
\begin{equation}\label{eq:AAt}
\sum_{l=1}^n a_i^l(\cdot, g)a_j^l(\cdot, g)=g_{ij}(\cdot) \qquad\text{or equivalently}\ A A^t= G,
\end{equation}
where $G$ has entries $g_{ij}(\cdot)$. Similarly, we have
\begin{equation}\label{eq:BBt}
\sum_{l=1}^n b_l^i(\cdot,g) b_l^j(\cdot,g)=g^{ij}(\cdot)\qquad\text{or equivalently} \
B^t B=G^{-1}.
\end{equation}
Taking the scalar product of the first equation in (\ref{eq:viedbein}) with $e_m$ yields after inserting the second equation
\begin{equation}\label{eq:BC}
a_i^l(\cdot, g)= \sum_{l=1}^n g_{ij}(\cdot) b_l^j(\cdot, g) \qquad \text{and} \qquad b_l^j(\cdot, g)= \sum_{l=1}^n g^{ij}(\cdot) a_i^l(\cdot, g).
\end{equation}
The second equation in (\ref{eq:BC}) is also derived by multiplying by the inverse of the metric.

 \begin{lem}\label{lem:vielbeincomp} At any point in $U_p$ and $\lambda>0$, we have
\[a_{i}^{l}(\cdot,g_\lambda)=a_{i}^{l}(\mathfrak{f}_\lambda(\cdot),g) \qquad\text{and}\qquad b_{l}^{i}(\cdot,g_\lambda)=b_{l}^{i}( {\mathfrak f}_\lambda(\cdot),g).\]
 \end{lem}

\begin{proof}
To prove the equality $a_{i}^{l}(\cdot,g_\lambda)=a_{i}^{l}(\mathfrak{f}_\lambda(\cdot),g)$, we let $\{{\bar e}_1(\cdot,g_\lambda),\dots, {\bar e}_n(\cdot,g_\lambda)\}$ denote the orthonormal frame obtained by parallel transport from $\{e_1(p),\dots,e_n(p)\}$ with respect to the metric $g_\lambda$. We know from Lemma~\ref{ref:eq:patrgglambda} that ${\bar e}_i(\cdot,g_\lambda)=e_i(\mathfrak{f}_\lambda(\cdot),g)$. Using \eqref{eq:viedbein} with respect to the metric $g_\lambda$, we have that
\[\frac{\partial}{\partial x^i}\circ \mathfrak{f}_\lambda=\sum_{l=1}^n a_i^l(\cdot,g_\lambda) \bar{e}_l(\cdot,g_\lambda).\]
Also, \eqref{eq:viedbein} applied to the point $\mathfrak{f}_\lambda(\cdot)$ gives that
\[\frac{\partial}{\partial x^i}\circ \mathfrak{f}_\lambda=\sum_{l=1}^n a_i^l(\mathfrak{f}_\lambda(\cdot),g) e_l(\mathfrak{f}_\lambda(\cdot),g).\]
Comparing both equations yields the result. The second equality can be proven in the same way.
\end{proof}

\section{Complex powers and logarithms of elliptic operators}\label{appendix2}
Let $E\to M$ be a vector bundle $E$ over $M$ of rank $k$. We consider an operator $Q$ in $\Psi_{\rm cl}(M, E)$ of positive real order $m$ with angle $\theta\in [0,2\pi)$. For $0<\delta$, we define the contour along the ray~$L_\theta$ around the spectrum of $Q$:
\[\Gamma_\theta=\Gamma_\theta^1\cup \Gamma_\theta^2\cup \Gamma_\theta^3 ,\]
where
\[\Gamma_\theta^1:=\bigl\{r{\rm e}^{{\rm i}\theta}, \delta< r\bigr\},\qquad \Gamma_\theta^2=\bigl\{\delta {\rm e}^{{\rm i}t}, \theta-2\pi< t<\theta\bigr\},\qquad \Gamma_\theta^3=\bigl\{r {\rm e}^{{\rm i}(\theta-2\pi)}, \delta< r\bigr\}.\]
For any operator $Q\in \Psi_{\rm cl}(M, E)$ with positive real order $m$ and Agmon angle $\theta$, the resolvent $\bigl(Q-\lambda\bigr)^{-1}$ is a bounded linear operator on $L^2(M, E)$ with operator norm $O(\vert \lambda\vert^{-1})$ as $\lambda$ tends to infinity in a sector around the contour $\Gamma_\theta$. For $\operatorname{Re}(z)<0$,
	the Cauchy integral \cite[equation~(1.5.7.1)]{Sc}
	\begin{equation*}
		Q_\theta^z:=\frac{{\rm i}}{2\pi}\int_{\Gamma_{\theta}} \lambda_\theta^z (Q-\lambda)^{-1} {\rm d}\lambda,\end{equation*}
	converges in the operator norm to a bounded linear operator on $L^2(M,E) $. Here $\lambda_\theta^z= \vert \lambda\vert^z {\rm e}^{{\rm i} z {\rm arg}_\theta\lambda}\!$, with ${\rm arg}_\theta\lambda$ the argument of $\lambda$ in $]\theta, \theta+2\pi[$.
For any real $s$, it also defines a linear operator $Q_\theta^z\colon H^s(M, E)\to H^s (M, E)$ on the Sobolev closure $H^s(M,E)$ of the space $C^\infty(M,E)$ of smooth sections of $E$ and, therefore, induces a linear operator on $C^\infty(M,E)$. Complex powers of elliptic operators $Q_\theta^z$ can be extended to any $z\in \C$ with $\operatorname{Re}(z)<k$ for any $k\in \N$ by
	\begin{equation}\label{eq:zmoinsk} Q_\theta^{z}:= Q^k Q_\theta^{z-k},\end{equation} thus giving rise to a group of well-defined complex powers $Q_\theta^z$ acting on $C^\infty(M,E)$ (see \cite[Section~1.5.7.1]{Sc})
	\[Q_\theta^0={\rm Id}, \quad Q_\theta^{z+w}=Q_\theta^{z} Q_\theta^{w},\qquad \forall (z,w)\in \C^2.\]
For $\operatorname{Re}(z)<0$, the symbol of $Q_\theta^{z}$ reads \cite[formula~(4.8.2.6)]{Sc}
\begin{equation*}	
\sigma(Q_\theta^{z})\sim\frac{{\rm i}}{2\pi} \int_{\Gamma_\theta} \lambda_\theta^{z} (\sigma(Q)-\lambda)^{*-1} {\rm d}\lambda,\end{equation*}
where $\sigma(Q)$ is the symbol of $Q$, the star stands for the inverse in the symbol algebra and
\begin{align*}
(\sigma(Q)-\lambda)^{*-1}&{}= \sigma_L(Q)^{*-1} \bigl(\sigma(Q)\sigma_L(Q)^{*-1} -\lambda \sigma_L(Q)^{*-1} \bigr)^{*-1}\\
&{}=\sigma_L(Q)^{-1} (1+\text{symbol of order} <0)^{*-1}
\end{align*}
is obtained by means of an expansion in $\xi$.
For $\operatorname{Re}(z)<0$, the complex power $Q_\theta^z$ is a classical pseudodifferential operator of order $mz$, whose symbol has the asymptotic expansion
\begin{equation} \label{eq:sigmaQzexp}\sigma(Q_\theta^z)(x,\xi)\sim \sum_{j=0}^\infty\sigma_{mz-j}(Q_\theta^z)(x,\xi),\end{equation}
where $\sigma_{mz-j}(Q_\theta^z)(x,\xi)$ are the positively homogeneous functions of degree $mz-j$ given by {\cite[formula (4.8.2.7)]{Sc}:
\begin{equation}\label{eq:symbolqz}
\sigma_{mz-j}(Q_\theta^z)(x,\xi)=\frac{{\rm i}}{2\pi}\int_{\Gamma_\theta}\lambda_\theta^z\sigma_{-m-j}\bigl((Q-\lambda)^{-1}\bigr)(x,\xi) {\rm d}\lambda.
\end{equation}
For a complex number $z$ with $\operatorname{Re}(z)<k$ and $k\in \N$, the complex power $Q_\theta^z$ given by (\ref{eq:zmoinsk}) is also a classical pseudodifferential operator of order $mz$ as a product of classical pseudodifferential operators, $Q^k$ of order~$mk$ and $Q_\theta^{z-k}$ of order $m(z-k)$. It has an asymptotic expansion as in~(\ref{eq:sigmaQzexp}) whose homogeneous components are given by the product formula for symbols \cite[formula~(10.16)]{Sh}: \begin{equation}\label{sigmajzk}
\sigma_{mz-j}(Q_\theta^z)= \sum_{k,a+b+\vert \alpha\vert=j} \frac{(-{\rm i})^{\vert \alpha \vert}}{\alpha!} \partial_\xi^\alpha \sigma_{mk-a} \bigl(Q^k\bigr) \partial_x^\alpha\sigma_{m(z-k)-b}\bigl(Q_\theta^{z-k}\bigr),\qquad \forall j\in \Z_{\geq 0}. 	\end{equation} }
	
	For $z_0\in \C$ with $\operatorname{Re}(z_0)< 0$ and following the notations of \cite[formula (2.6.1.6)]{Sc}, we consider the Cauchy integral
	\begin{gather*}
		L_\theta(Q, z_0)
		:=\frac{{\rm i}}{2\pi}\int_{\Gamma_{\theta}} \log _\theta(\lambda) \lambda_\theta^{z_0} (Q-\lambda)^{-1} {\rm d}\lambda,
	\end{gather*}
	which is absolutely convergent and defines an bounded linear operator $ L_\theta(Q, z_0)\colon H^s(M,E)\longrightarrow H^{s-m\operatorname{Re}(z_0)-\epsilon}(M,E)$ for any real number $s$ and for any positive $\epsilon$.
	As we did for complex powers, we extend $ L_\theta(Q, z_0)$ to $\operatorname{Re}(z_0)<k$ for any $k\in \N$ by \cite[formula (2.6.1.7)]{Sc}
	\begin{equation*}	
		L_\theta(Q, z_0):= L_\theta(Q, z_0-k) Q^k. \end{equation*}
	thus defining $L_\theta(Q, z_0)$ for any $z_0\in \C$ such that
	\[L_\theta(Q, z_0)= L_\theta(Q, 0) Q_\theta^{z_0}= \log_\theta(Q) Q_\theta^{z_0}.\]
Here we have set $\log_\theta(Q):= L_\theta(Q, 0),$ which defines a bounded linear operator
\[
L_\theta(Q, z_0)\colon \ H^s(M,E)\to H^{s -\epsilon}(M,E)
\]
 for any real number $s$ and any $\epsilon>0$. Since $\log_\theta(\lambda) \lambda_\theta^z= \lambda_\theta^z \log_\theta (\lambda),$ extending instead by
	\begin{equation*} L_\theta(Q, z_0):= Q^k L_\theta(Q, z_0-k), \end{equation*} gives rise to the same family of operators and $\log _\theta (Q) Q_\theta^z=Q_\theta^z \log _\theta (Q)$ for any complex number~$z$.

\section[The supertrace versus the Berezin integral and the Getzler rescaling]{The supertrace versus the Berezin integral\\ and the Getzler rescaling} \label{sec:appendix3}

Following \cite[Section~3]{BGV}, we review the construction and properties of the Berezin integral together with its relation with the supertrace as well as Geztler's rescaling on differential forms. Let $V$ be a real vector space. A linear map $T\colon\Lambda V\to \mathbb{R}$ is called a Berezin integral if $T$ vanishes on~$\Lambda^kV$ for $k<n={\rm dim}(V)$. If $V$ is an oriented Euclidean vector space equipped with a scalar product~$g$, there exist a canonical Berezin integral defined as the projection of any element in~$\Lambda V$ onto its component on the $n$-form ${\rm e}^1\wedge \dots \wedge {\rm e}^n$ where $\{e_1,\dots,e_n\}$ is an orthonormal basis of~$(V,g)$. We will denote this Berezin integral by $T$
\[T\bigl({\rm e}^1\wedge \dots \wedge {\rm e}^n\bigr)=1,\qquad T\bigl({\rm e}^I\bigr)=0\qquad \text{if}\ |I|<n.\]
Let us now consider the Clifford algebra $\Cl(V)$ of $(V,g)$. The isometry $-{\rm Id}\in O(V,g)$ gives rise to the map
\begin{align*}
	\Phi\colon \ \Cl(V) &\longrightarrow \Cl(V),\nonumber\\ {\rm e}^{i_1}\cdot_g\cdots\cdot_g {\rm e}^{i_k}&\longmapsto (-1)^k{\rm e}^{i_1}\cdot_g\cdots\cdot_g {\rm e}^{i_k}
\end{align*}
with $i_1<\dots< i_k$. Here ``$\cdot_g$'' denotes the Clifford multiplication with respect to the metric~$g$. The map $\Phi$ clearly satisfies $\Phi^2={\rm Id}$. Therefore, we get a splitting of $\Cl(V)$ into
\[\Cl(V)=\Cl(V)^+\oplus \Cl(V)^-,\]
where $\Cl(V)^\pm:=\{a\in \Cl(V)\mid \Phi(a)=\pm a\}$. Now we have the proposition \cite[Proposition~3.19]{BGV}.

\begin{prop} \label{pro:spinorbundle}
Let $V$ be an oriented Euclidean vector space of even dimension~$n$. There exists a unique $\mathbb{Z}_2$-graded Clifford module $S=S^+\oplus S^-$, such that
	\[\Cl(V)\otimes \C\simeq \operatorname{End}(S).\]
	In particular, ${\rm dim} S^\pm=2^{\frac{n}{2}-1}$. Also, we have that $\Cl(V)^+\cdot S^\pm\subset S^\pm$ and $\Cl(V)^-\cdot S^\pm\subset S^\mp$. Therefore, we have the isomorphism \[\Cl(V)^+\otimes \C\simeq \operatorname{End}\bigl(S^\pm\bigr).\]
\end{prop}
Notice that $S^\pm$ are defined as the eigenspaces of $S$ associated with the eigenvalues $\pm 1$ of the complex volume form $\omega_{\mathbb{C}}={\rm i}^{\frac{n}{2}} {\rm e}^1\cdot_g\dots\cdot_g {\rm e}^n$ with $n$ even. Now, there is a natural super trace on $\Cl(V)\otimes \mathbb{C}$, defined by
\begin{equation*}
	{\rm str}(a) :=
	\begin{cases}
		{\rm tr}^{S^+}(a)-{\rm tr}^{S^-}(a) & \text{if $a\in \Cl(V)^+$,}\\
		0 & \text{if $a \in \Cl(V)^-,$}
	\end{cases}
\end{equation*}
where, as before, ${\rm tr}^E$ stands for the fibrewise trace on $\operatorname{End}(E)$. In order to relate this supertrace with the Berezin integral, we assign to any vector $v\in V^*$ the endomorphism $c^g(v)\in \operatorname{End}(\Lambda V) $ that uniquely extends to a morphism of algebra bundles
\begin{equation}\label{eq:cliffend}
c^g\colon \ \Cl(V) \longrightarrow \operatorname{End}(\Lambda V)
\end{equation}
defined by
$	
	c^g(v)\bullet =v\wedge\bullet -v^{\sharp_g} \lrcorner\bullet$,
where $v^{\sharp_g}$ is the vector in $V$ associated to its covector $v$ by the musical isomorphism. The symbol map ${\bf s}^g\colon \Cl(V)\to \Lambda V$ is the isomorphism given by ${\bf s}^g(a):=c^g(a)1$. Indeed, to show ${\bf s}^g$ (as well as $c^g$) is injective, we assume that ${\bf s}^g(a)=0$ for some $a=\sum_{i_1<\dots<i_k} a_{i_1\dots i_k} {\rm e}^{i_1}\cdot_g\cdots\cdot_g {\rm e}^{i_k}$. Then $c^g(a)1=\sum_{i_1<\dots<i_k} a_{i_1\dots i_k} {\rm e}^{i_1}\wedge\dots\wedge {\rm e}^{i_k}=0$. Hence, $a_{i_1\dots i_k}=0$ and, thus, $a=0$. The bijectivity of ${\bf s}^g$ comes from the equality of the dimensions. In \cite[Proposition~3.21]{BGV}, it is shown that there is a~unique supertrace which is related to the Berezin integral by the following
\begin{equation}\label{eq:supertraceformula}
	{\rm str}(a)=(-2{\rm i})^{\frac{n}{2}}(T\circ {\bf s}^g)(a).
\end{equation}
for any $a\in \Cl(V)\otimes \mathbb{C}$.

\subsection*{Acknowledgements} The first named author would like to thank the Alfried Krupp Wissenschaftskolleg in Greifswald for the support. We are grateful to the Humboldt Foundation for funding a Linkage Programm between the University of Potsdam in Germany and the Lebanese University, as well as the American University of Beirut in Lebanon. We also thank the referees for their very helpful comments.

\pdfbookmark[1]{References}{ref}
\LastPageEnding


\begin{thebibliography}{99}
\footnotesize\itemsep=0pt

\bibitem{AGHM}
Ammann B., Grosjean J.F., Humbert E., Morel B., A spinorial analogue of
 {A}ubin's inequality, \href{https://doi.org/10.1007/s00209-007-0266-5}{\textit{Math.~Z.}} \textbf{260} (2008), 127--151,
 \href{https://arxiv.org/abs/math.DG/0308107}{arXiv:math.DG/0308107}.

\bibitem{ABP}
Atiyah M., Bott R., Patodi V.K., On the heat equation and the index theorem,
 \href{https://doi.org/10.1007/BF01425417}{\textit{Invent. Math.}} \textbf{19} (1973), 279--330.

\bibitem{BGV}
Berline N., Getzler E., Vergne M., Heat kernels and {D}irac operators,
 \textit{Grundlehren Text Ed.}, Springer, Berlin, 2004.

\bibitem{BHMMM}
Bourguignon J.-P., Hijazi O., Milhorat J.-L., Moroianu A., Moroianu S., A
 spinorial approach to {R}iemannian and conformal geometry, \textit{EMS Monogr. Math.},
 \href{https://doi.org/10.4171/136}{European Mathematical Society (EMS)}, Z\"urich, 2015.

\bibitem{C}
Connes A., Noncommutative geometry, Academic Press, Inc., San Diego, CA, 1994.

\bibitem{DS1}
Debord C., Skandalis G., Adiabatic groupoid, crossed product by
 {$\mathbb{R}_+^\ast$} and pseudodifferential calculus, \href{https://doi.org/10.1016/j.aim.2014.02.012}{\textit{Adv. Math.}}
 \textbf{257} (2014), 66--91, \href{https://arxiv.org/abs/1307.6320}{arXiv:1307.6320}.

\bibitem{DS2}
Debord C., Skandalis G., Blow-up constructions for {L}ie groupoids and a
 {B}outet de {M}onvel type calculus, \href{https://doi.org/10.17879/59019640550}{\textit{M\"unster~J. Math.}} \textbf{14}
 (2021), 1--40, \href{https://arxiv.org/abs/1705.09588}{arXiv:1705.09588}.

\bibitem{E}
Epstein D.B.A., Natural tensors on {R}iemannian manifolds,
 \href{https://doi.org/10.4310/jdg/1214433166}{\textit{J.~Differential Geometry}} \textbf{10} (1975), 631--645.

\bibitem{F}
Freed D., Lectures on Dirac operators, {U}npublished notes, 1987, available at
 \url{https://web.ma.utexas.edu/users/dafr/DiracNotes.pdf}.

\bibitem{Ge}
Getzler E., A short proof of the local {A}tiyah--{S}inger index theorem,
 \href{https://doi.org/10.1016/0040-9383(86)90008-X}{\textit{Topology}} \textbf{25} (1986), 111--117.

\bibitem{G}
Gilkey P.B., Invariance theory, the heat equation, and the {A}tiyah--{S}inger
 index theorem, 2nd~ed., \textit{Stud. Adv. Math.}, \href{https://doi.org/10.1201/9780203749791}{CRC Press}, Boca Raton, FL, 1995.

\bibitem{H}
Higson N., The tangent groupoid and the index theorem, in Quanta of Maths,
 \textit{Clay Math. Proc.}, Vol.~11, American Mathematical Society,
 Providence, RI, 2010, 241--256.

\bibitem{HY}
Higson N., Yi Z., Spinors and the tangent groupoid, \href{https://doi.org/10.4171/DM/712}{\textit{Doc. Math.}}
 \textbf{24} (2019), 1677--1720, \href{https://arxiv.org/abs/1902.08351}{arXiv:1902.08351}.

\bibitem{LM}
Lawson Jr.~H.B., Michelsohn M.L., Spin geometry, \textit{Princeton Math. Ser.},
 Vol.~38, Princeton University Press, Princeton, NJ, 1989.

\bibitem{MP}
Mickelsson J., Paycha S., The logarithmic residue density of a generalized
 {L}aplacian, \href{https://doi.org/10.1017/S144678871100108X}{\textit{J.~Aust. Math. Soc.}} \textbf{90} (2011), 53--80,
 \href{https://arxiv.org/abs/1008.3039}{arXiv:1008.3039}.

\bibitem{Sc}
Scott S., Traces and determinants of pseudodifferential operators, \textit{Oxford Math.
 Monogr.}, \href{https://doi.org/10.1093/acprof:oso/9780198568360.001.0001}{Oxford University Press}, Oxford, 2010.

\bibitem{Sh}
Shubin M.A., Pseudo-differential operators and spectral theory, \href{https://doi.org/10.1007/978-3-642-56579-3}{Springer},
 Berlin, 2001.

\bibitem{vEY}
van Erp E., Yuncken R., A groupoid approach to pseudodifferential calculi,
 \href{https://doi.org/10.1515/crelle-2017-0035}{\textit{J.~Reine Angew. Math.}} \textbf{756} (2019), 151--182,
 \href{https://arxiv.org/abs/1511.01041}{arXiv:1511.01041}.

\bibitem{W}
Wodzicki M., Noncommutative residue. {I}. {F}undamentals, in {$K$}-theory,
 Arithmetic and Geometry ({M}oscow, 1984--1986), \textit{Lecture Notes in
 Math.}, Vol.~1289, \href{https://doi.org/10.1007/BFb0078372}{Springer}, Berlin, 1987, 320--399.

\end{thebibliography}
\end{document}